\documentclass[12pt,a4paper]{amsart}
\usepackage[letterpaper, margin=0.8in]{geometry}
\usepackage{amsfonts,amssymb,amsthm,epsf, graphics, bbm, mathabx, verbatim, caption, subcaption, enumerate}
\usepackage{comment}

\usepackage{indentfirst}

\usepackage{tikz}
\usetikzlibrary{patterns}

\newcommand{\ZZ}{\mathbb{Z}}

\newcommand{\RR}{\mathbb{R}}
\newcommand{\CC}{\mathbb{C}}
\newcommand{\NN}{\mathbb{N}}

\newcommand{\norm}[1]{\left\lVert#1\right\rVert}

\newcommand*\conj[1]{\overline{#1}}

\newtheorem{thm}{\bf Theorem}
\newtheorem{theorem}{Theorem}[section]

\newtheorem{proposition}[theorem]{Proposition}
\newtheorem{lemma}[theorem]{Lemma}

\newtheorem{conjecture}[theorem]{Conjecture}

\theoremstyle{remark}
\newtheorem*{remark}{Remark}

\theoremstyle{definition}
\newtheorem{definition}[theorem]{Definition}

\usepackage{mathtools}

\title{Fourier dimension of conical and cylindrical hypersurfaces}
\author{Junjie Zhu
}

\newcommand{\Addresses}{{
  \bigskip
  \footnotesize

  \textsc{Department of Mathematics, 1984 Mathematics Road, University of British Columbia, Vancouver, BC Canada V6T 1Z2}\par\nopagebreak
  \textit{E-mail address}: \texttt{jzhu@math.ubc.ca}
}}
\date{\today}


\begin{document}

\maketitle

\begin{abstract}
     The notions of Hausdorff and Fourier dimensions are ubiquitous in harmonic analysis and geometric measure theory. It is known that any hypersurface in $\RR^{d+1}$ has Hausdorff dimension $d$. However, the Fourier dimension depends on the finer geometric properties of the hypersurface. For instance, the Fourier dimension of a hyperplane is 0, and the Fourier dimension of a hypersurface with non-vanishing Gaussian curvature is $d$. Recently, Harris has shown that the Euclidean light cone in $\RR^{d+1}$ has Fourier dimension $d-1$, which leads one to conjecture that the Fourier dimension of a hypersurface equals the number of non-vanishing principal curvatures. We prove this conjecture for all $d$-dimensional cones and cylinders in $\RR^{d+1}$ generated by hypersurfaces in $\mathbb{R}^d$ with non-vanishing Gaussian curvature. In particular, cones and cylinders are not Salem. Our method involves substantial generalizations of Harris's strategy.
\end{abstract}

\section{Introduction}

\subsection{Hausdorff and Fourier dimensions}

Finding the sizes of sets is one of many fundamental questions in mathematics. Depending on the context, we measure the sets with different notions of size. We use cardinality to measure finite sets, and the Lebesgue measure for Borel sets in $\RR^{n}$. These notions of sizes are less effective on fractal sets. 

For example, for $\alpha<1$, the middle-$\alpha$ Cantor set, which is constructed by starting with the unit interval $[0, 1] \subset \RR$ and repeatedly removing the middle-$\alpha$ portions of intervals, is uncountable but has Lebesgue measure $0$. Nevertheless, we expect that smaller $\alpha$ corresponds to a larger middle-$\alpha$ set. Thus, we need new notions to describe sizes of sets like the middle-$\alpha$ Cantor set. 

One notion often used in fractal geometry is the \textit{Hausdorff dimension} defined as follows. For a set $A \subset \RR^n$, $s, \delta > 0$, 
$$\mathcal{H}^{s}_{\delta}(A) := \inf \{\sum_{j} \text{diam}(E_j)^s | A \subset \cup_{j} E_j, \text{diam}(E_j) < \delta\},$$
and the \textit{$s$-dimensional Hausdorff measure} $\mathcal{H}^{s}(A) := \lim_{\delta \to 0} \mathcal{H}^{s}_{\delta}(A)$. The \textit{Hausdorff dimension} of $A$ is $\dim_H(A) := \sup \{s : \mathcal{H}^{s}(A) = \infty\} = \inf \{s: \mathcal{H}^{s}(A) = 0 \}.$

Frostman's lemma \cite[Theorem 2.7]{mattila_2015} offers an alternative characterization of the Hausdorff dimension. Let $\mathcal{M}(A)$ be the set of measure $\mu$ supported on $A$ with finite total mass $\norm{\mu}_{1} := \mu(A) < \infty$. The \textit{Fourier transform} of a measure $\mu$ at $\xi \in \RR^{n}$ is defined as $\widehat{\mu}(\xi) := \int e^{-2\pi i x \xi} d\mu(x)$. For a Borel set $A \subset \RR^{n}$, 
$$\dim_H(A) = \sup \{s \in [0, n]: \exists \mu \in \mathcal{M}(A), I_s(\mu) := \int |\widehat{\mu}(\xi)|^2 |\xi|^{s-n} d\xi < \infty \},$$ where $I_s$ is the $s$\textit{-energy} of $\mu$.

For $\mu \in \mathcal{M}(A)$, if $\sup_{\xi \in \RR^{n}}|\xi|^{\frac{s}{2}}|\widehat{\mu}(\xi)| < \infty$ for an $s > 0$, $I_{s_0}(\mu) < \infty$ when $0<s_0 < s$. From the characterization above, $\dim_H(A) \geq s$. This motivates the notion of \textit{Fourier dimension}, defined as 
$$\dim_F(A) := \sup \{s \in [0, n]: \exists \mu \in \mathcal{M}(A), \sup_{\xi \in \RR^{n}} |\xi|^{\frac{s}{2}}|\widehat{\mu}(\xi)| < \infty  \}$$
for Borel $A \subset \RR^{n}$. Consequently, $\dim_F(A) \leq \dim_H(A)$, and the inequality is strict for some sets $A$. The set $A$ is \textit{Salem} if $\dim_F(A)=\dim_H(A)$. 

An interesting question is what properties of the set the Fourier dimension captures, and whether we can infer that the set has some properties from its Fourier and Hausdorff dimensions. A related task is to identify the similarities and differences between the two dimensions. For example, there are properties enjoyed by sets of large Fourier dimensions, but not shared by sets only with large Hausdorff dimensions. A work of this genre is \cite{liang} by Liang and Pramanik. 

One path to study the questions above is to determine the Fourier dimension of a given set $A$ and relate the Fourier dimension with properties of $A$. Other paths include identifying Salem sets and constructing sets with Fourier dimension $\alpha$ and often with Hausdorff dimension $\beta$, where $0 \leq \alpha \leq \beta \leq n$. We refer readers to \cite{survey}, \cite{ThomasWilliamKörner2011}, and \cite{lp} for samples of work.

\subsection{Orientable hypersurfaces}

Let $M\subset \RR^{d+1} $ be an orientable hypersurface, which is a topological manifold of dimension $d$ with a normal direction $N: M \to \mathbb{S}^{d}$. This manuscript focuses on computing $\dim_F(M)$, interpreting the properties of $M$ that $\dim_F(M)$ captures, and differentiating between topological, Hausdorff, and Fourier dimensions on $M$.

Since the Hausdorff dimension is preserved under diffeomorphisms, which are bi-Lipschitz, $\dim_H(M) = d$ \cite[Proposition 3.1]{falconer}. In contrast, in many cases, the Fourier dimension of $M$ depends on the \textit{principal curvatures} and the \textit{Gaussian curvature}, which are eigenvalues and the determinant of the \textit{Weingarten map}, respectively. We refer readers to section \ref{sec62} for definitions and more discussions on the Weingarten map, principal curvatures, and the Gaussian curvature. If $M$ is a hyperplane, all principal curvatures at any point $p \in M$ are $0$, and $\dim_F(M) = 0$. If all principal curvatures of points $p \in M$ are non-zero, $\dim_F(M) = d$ by the following theorem. 

\begin{thm}\cite[{ Section 8.3, Theorem 1}]{Stein1993}
\label{th191}
    Let $M \subset \RR^{d+1}$ be a compact orientable hypersurface with non-vanishing Gaussian curvature. Then there exists $\mu \in \mathcal{M}(M)$ with $\sup_{\xi \in \RR^{d+1}}|\xi|^{\frac{d}{2}}|\widehat{\mu}(\xi)| < \infty$.
\end{thm}

The theorem above inspires a natural conjecture of relating the Fourier dimension of $M$ with the number of non-zero principal curvatures. 

\begin{conjecture}
    Let $M$ be a $d$-dimensional orientable hypersurface in $\RR^{d+1} $. Suppose $k$ of the $d$ principal curvatures are not vanishing, and the other $d-k$ principal curvatures are always zero. Then $\dim_F(M) = k$.
\end{conjecture}

If the conjecture holds, we can view the Fourier dimension as a generalization of the number of non-zero principal curvatures only defined on manifolds, and it provides an avenue for introducing concepts related to curvatures to fractal sets.

The conjecture holds when $k=0$ ($M$ is a hyperplane), and $k=d$ (Theorem \ref{th191}). By Harris et al., the conjecture holds when the manifold is the light cone $\mathcal{C}$, where the number of non-zero principal curvatures at any point on $\mathcal{C} \backslash \{ 0\}$ is $d-1$.
The proof relies on the rotational symmetry of $\mathcal{C}$.

\begin{thm}[\cite{Fraser_2022}]
\label{t16}
    If $\mathcal{C}$ is the light cone 
$$\mathcal{C} = \{ h(x, 1) | x \in \mathbb{S}^{d-1}, h \in \RR \},$$
then $\dim_F(\mathcal{C})=d-1$.
\end{thm}

\subsection{Main results}

We generalize Theorem \ref{t16} to more general cones, adapt the method to cones without rotational symmetry, and show the conjecture also holds in this case.
\begin{theorem}
\label{t14}
    Let $S\subset \RR^d$ be a hypersurface with $\text{Int} S$, the interior of $S$, having non-vanishing Gaussian curvature, and the cone generated by $S$ is
\begin{equation*}
    C := \{h(x, 1) |x \in S, h \in \RR \} \subset \RR^{d+1}.
\end{equation*}
Then $\dim_F(C)=d-1$.
\end{theorem}

\begin{remark}
    \begin{enumerate}
        \item At any point on $C\backslash\{0\}$, the number of non-zero principal curvatures is $d-1$. In addition, diffeomorphisms in general do not preserve Fourier dimension. Therefore, Theorem \ref{t16} does not imply Theorem \ref{t14} directly.
        \item The above result also suggests that for the cone generated by $S$ of non-vanishing Gaussian curvature, rotational symmetry does not affect its Fourier dimension.
    \end{enumerate}
\end{remark}  

We have a similar conclusion for cylinders, where the number of non-zero principal curvatures at any point on the cylinders is also $d-1$.
\begin{theorem}
\label{t120}
    Let $S\subset \RR^d$ be a hypersurface with $\text{Int} S$, the interior of $S$, having non-vanishing Gaussian curvature, and the cylinder generated by $S$ is
\begin{equation*}
    D := S \times \RR = \{(x, h) |x \in S, h \in \RR \} \subset \RR^{d+1}.
\end{equation*}
Then $\dim_F(D)=d-1$.
\end{theorem}

To establish Theorems \ref{t16}, \ref{t14}, and \ref{t120}, we need to show the Fourier dimension is bounded below and above by $d-1$. For the upper bound, the main strategy in all cases is to create a new measure by averaging a series of push-forwards of an old measure. The push-forwards used in the light cone are based on the \textit{orthogonal group} $O(d)$, which is the group of automorphisms on the sphere $\mathbb{S}^{d-1}$. An adaptation is needed for other conical hypersurfaces as not every element in $O(d)$ maps at least a compact submanifold of $S$ to $S$. We can further modify the strategy in the case of cylinders.

\subsection{Organization of the manuscript}
Theorem \ref{t14} follows from propositions \ref{p31} and \ref{p32}, and Theorem \ref{t120} follows from propositions \ref{p33} and \ref{p34}. We show Propositions \ref{p31} and \ref{p33} that give lower bounds on the Fourier dimension in section \ref{sec4}. We present examples of Propositions \ref{p32} and \ref{p34} in section \ref{sec267}, then we prove Proposition \ref{p32} in section \ref{sec764} and Proposition \ref{p34} in section \ref{sec981}. The proofs of Propositions \ref{p32} and \ref{p34} rely on a technical lemma presented in section \ref{sec_ex}. Finally, we propose a new question in section \ref{sec_nq}.

\subsection{Acknowledgements} This project originated from a student research seminar in the winter term of 2023 at the University of British Columbia. I thank my advisor Malabika Pramanik for the discussion and the guidance in preparing this manuscript.

\section{Proof of Theorems \ref{t14} and \ref{t120}}

Theorem \ref{t14} is a corollary of propositions \ref{p31} and \ref{p32}, and theorem \ref{t120} follows from propositions \ref{p33} and \ref{p34}.
\begin{proposition}
\label{p31}
    $\dim_F(C) \geq d-1$. 
\end{proposition}

\begin{proposition}
\label{p32}
    $\dim_F(C) \leq d-1$.
\end{proposition}

\begin{proposition}
\label{p33}
    $\dim_F(D) \geq d-1$. 
\end{proposition}

\begin{proposition}
\label{p34}
    $\dim_F(D) \leq d-1$.
\end{proposition}

\begin{remark}
The following theorem by Littman, which comes from modifying the proof of Theorem \ref{th191}, implies that $\dim_F(M) \geq k$ for general hypersurfaces $M \subset \RR^{d+1}$ with $k$ non-zero principal curvatures.

\begin{thm}[\cite{Littman}]
\label{t33}
    Suppose at least $k$ of the principal curvatures are not zero at all points $p \in M$. There exists a measure $\mu \in \mathcal{M}(M)$ such that $\sup_{\xi \in \RR^{d+1}}|\xi|^{\frac{k}{2}}|\widehat{\mu}(\xi)| < \infty $.
\end{thm}

Propositions \ref{p31} and \ref{p33} can be concluded from Theorem \ref{t33} since the theorem implies all orientable hypersurfaces with $d-1$ non-zero principal curvatures, including cones and cylinders, have Fourier dimension at least $d-1$. Nonetheless, we include simpler proofs that provide lower bounds on the Fourier dimensions of cones and cylinders.
\end{remark}

\section{Lower bounds on the Fourier dimensions}
\label{sec4}

In this section, we provide lower bounds on the Fourier dimensions of cones and cylinders by proving Propositions \ref{p31} and \ref{p33}. The idea is to construct measures with the needed Fourier decay.

\begin{proof}[Proof of Proposition \ref{p31}]
    By Theorem \ref{th191}, there exist a measure $\mu^{(S)} \in \mathcal{M}(S)$ and $c_{\mu} > 0$, such that $|\widehat{\mu^{(S)}}(\xi)| \leq c_{\mu}|\xi|^{-\frac{d-1}{2}}$. Let $\psi \in C_0^{\infty}(\RR)$ with $\text{spt } \psi \subset [1, 2]$, $\psi \geq 0$, and $\int \psi(h) dh=1$. For non-negative Borel functions $f$ on $\RR^{d+1}$, we define $\nu^{(C)} \in \mathcal{M}(C)$ such that 
\begin{equation*}
    \int f d\nu^{(C)} : = \int \psi(h) \int f(h(y, 1)) d\mu^{(S)}(y)dh.
\end{equation*}
We claim $\nu^{(C)}$ is well defined and supported on $C$. 
\begin{itemize}
    \item Let $T: L^{\infty}(\RR^{d+1}) \to \RR$ be 
    $$Tf := \int \psi(h) \int f(h(y, 1)) d\mu^{(S)}(y)dh.$$
    If $ \norm{f}_{L^{\infty}(\RR^{d+1})} = 1$,
    \begin{equation*}
        \begin{split}
            \left|Tf\right| & \leq  \int \psi(h) \int |f(h(y, 1))| d\mu^{(S)}(y)dh \\
            & \leq \int \psi(h) \int d\mu^{(S)} (y)dh \\
            & \leq 1.
        \end{split}
    \end{equation*}
    In addition, the map $T$ is linear. Therefore, $T$ is a bounded linear functional, and there is a unique measure $\nu^{(C)}$ on $\RR^{d+1}$ such that $Tf = \int_{\RR^{d+1}} f d\nu^{(C)}$. By letting $f=1$, we note that $\nu^{(C)}$ is a probability measure.
    \item Let $z \in C^c$, then there is a neighborhood $U$ such that $z \in U \subset C^c$. Therefore, $\mathbf{1}_{U}(h(y, 1))=0$ for $y \in S$, $h \in \RR$. Then $\nu^{(C)}(U) = 0$, so that $z \in U \subset (\text{spt } \nu^{(C)})^c$.
\end{itemize}

We write the direction $\eta= (\eta', \eta_{d+1}) \in \RR^{d} \times \RR$ with $|\eta|=1$. There are two cases to consider.

Case 1: $|\eta'| \leq c$ for a small constant $c>0$, then for all $N \in \NN$,
\begin{equation*}
    \begin{split}
        |\widehat{\nu^{(C)}}(k\eta)| &=  \left|\int \psi(h) \int e^{-2\pi i h(y, 1) \cdot k(\eta', \eta_{d+1})} d\mu^{(S)}(y) dh\right| \\
        & \leq   \int \left| \int \psi(h) e^{-2\pi i hk(y, 1) \cdot (\eta', \eta_{d+1})} dh \right|  d\mu^{(S)}(y) \\
        & \leq C_N  \int \left|k(y, 1)\cdot(\eta', \eta_{d+1})  \right|^{-N} d\mu^{(S)}(y), 
    \end{split}
\end{equation*}
where $C_N = \sup_{\xi \in \RR} |\xi|^N |\widehat{\psi}(\xi)|$. Suppose $S \subset B(0, R) \subset \RR^{d}$, then $|(y, 1)\cdot (\eta', \eta_{d+1})| \geq |\eta_{d+1}| - |y\cdot \eta'|\geq \sqrt{1-c^2}-Rc$. We may choose $c$ small enough so that 
$$|(y, 1)\cdot (\eta', \eta_{d+1})| \geq 2^{-1}.$$ So in this case,
$$|\widehat{\nu^{(C)}}(k\eta)| \leq 2^N C_N k^{-N}.$$

Case 2: $|\eta'| \geq c$, then
\begin{equation*}
    \begin{split}
        |\widehat{\nu^{(C)}}(k\eta)| = & \left| \int \psi(h) \int e^{-2\pi i h(y, 1) \cdot k (\eta', \eta_{d+1})} d\mu^{(S)} (y) dh\right| \\
        = & \left|\int e^{-2 \pi i h k \eta_{d+1}} \psi(h) \int e^{-2\pi i h ky \cdot \eta'} d\mu^{(S)}(y) dh\right| \\
        \leq & \int  \left|\psi(h) \widehat{\mu^{(S)}}(hk\eta')\right|  dh\\
        \leq  & c_{\mu} \int |\psi(h)| |hk\eta'|^{-\frac{d-1}{2}}  dh\\
        \leq &  c_{\mu} (ck)^{-\frac{d-1}{2}}
    \end{split}
\end{equation*}
since $h \in [1, 2]$ and $|\eta'| \geq c$.
\end{proof}

We have a similar proof for the cylinder $D$.

\begin{proof}[Proof of Proposition \ref{p33}]
    By Theorem \ref{th191}, there exist a measure $\mu^{(S)} \in \mathcal{M}(S)$ and $c_{\mu} > 0$, such that $|\widehat{\mu^{(S)}}(\xi)| \leq c_{\mu}|\xi|^{-\frac{d-1}{2}}$. Let $\psi \in C_0^{\infty}(\RR)$ with $\psi \geq 0$, and $\int \psi(h) dh=1$. We define $\nu^{(D)} \in \mathcal{M}(D)$ as a product measure
\begin{equation*}
    d \nu^{(D)} := d\mu^{(S)} \psi dh.
\end{equation*}

We write the direction $\eta= (\eta', \eta_{d+1}) \in \RR^{d} \times \RR$ with $|\eta|=1$. There are two cases to consider.

Case 1: $|\eta'| \leq c$ for a small constant $c>0$, then for all $N \in \NN$,
\begin{equation*}
    \begin{split}
        |\widehat{\nu^{(D)}}(k\eta)| &=  \left|\int  \int e^{-2\pi i (y, h) \cdot k(\eta', \eta_{d+1})} d\mu^{(S)}(y) \psi(h) dh\right| \\
        & \leq   \int \left| e^{-2 \pi i k y \cdot \eta'} \int  e^{-2\pi i h k   \eta_{d+1}} \psi(h) dh \right|   d\mu^{(S)}(y) \\
        & \leq C_N  \int \left|k\eta_{d+1} \right|^{-N} d\mu^{(S)}(y), 
    \end{split}
\end{equation*}
where $C_N = \sup_{\xi \in \RR} |\xi|^N |\widehat{\psi}(\xi)|$. As $|\eta_{d+1}| \geq \sqrt{1-c^2}$, by choosing $c$ such that $\sqrt{1-c^2} \geq 2^{-1}$,
$$|\widehat{\nu^{(C)}}(k\eta)| \leq 2^N C_N k^{-N}.$$

Case 2: $|\eta'| \geq c$, then
\begin{equation*}
    \begin{split}
        |\widehat{\nu^{(D)}}(k\eta)| = & \left| \int  \int e^{-2\pi i (y, h) \cdot k (\eta', \eta_{d+1})} d\mu^{(S)} (y) \psi(h) dh\right| \\
        = & \left|\int e^{-2 \pi i h k \eta_{d+1}} \psi(h) \int e^{-2\pi i yk \cdot \eta'} d\mu^{(S)}(y) \psi(h) dh\right| \\
        \leq & \int  \left|\psi(h) \widehat{\mu^{(S)}}(k\eta')\right|  dh\\
        \leq  & c_{\mu} \int |\psi(h)| |k\eta'|^{-\frac{d-1}{2}}  dh\\
        \leq &  c_{\mu} (ck)^{-\frac{d-1}{2}}
    \end{split}
\end{equation*}
since $|\eta'| \geq c$.
\end{proof}

\section{Upper bounds on the Fourier dimension: Examples}
\label{sec267}

Harris's method heavily uses the rotational symmetry for the Euclidean light cone. In this section, we analyze two cones and one cylinder in 3 dimensions that do not have this property.

\begin{definition}
\label{d21}
    Let $j$ be the index of hypersurfaces. Let $I_{j} \subset \RR$ be an interval containing $0$, and $$S_{j}: = \{(\gamma_j(x), x^2), x \in I_{j}\} \subset \RR^2,$$
where $\gamma_j: I_{j} \to \RR$ is smooth. A coordinate map of $S_{j}$ is $\phi_{j}: I_{j} \to \RR^2$, where $$\phi_{j}(x) := (\gamma_j(x), x^2).$$ A coordinate map of the cone $C_{j}$ is $\Phi_{j}: I_{j} \times \RR \to \RR^3$ with 
\begin{equation}
\label{eq117}
    \Phi_{j}(x, h) : = h(\gamma_j(x), x^2, 1) = h(\phi_{j}(x), 1).
\end{equation}

The two cones are:
\begin{enumerate}[a.]
    \item parabolic cone $C_{1}$, with $I_{1} = \RR$, $\gamma_1 : \RR \to \RR^2$ is $\gamma_1(x) := x$.
    \item a perturbed parabolic cone $C_{2}$, where $I_{2}$ is a sufficiently small neighborhood of $0$, $\gamma_2: I_{2} \to \RR^2$ is $\gamma_2(x):=x+x^3$.
\end{enumerate}

The parabolic cylinder $D_1$ is also generated by $S_1$, with a coordinate map $\widetilde{\Phi}_{1}: \RR \times \RR \to \RR^3$ given by
\begin{equation}
\label{eq286}
    \widetilde{\Phi}_{1}(x, h) : = (x, x^2, h).
\end{equation}

\end{definition}

The curve $S_{j}$ has non-vanishing Gaussian curvatures for $j = 1, 2$. We also note that the parabolic cone $C_{1}$ can be obtained from scaling and rotating the Euclidean light cone. We aim to show the following theorems in this section.
\begin{theorem}
\label{t2}
\begin{enumerate}
    \item $\dim_F(C_{j})=1$ for $j=1, 2$.
    \item $\dim_F(D_1) = 1$.
\end{enumerate}
    
\end{theorem}

We will show that an upper bound on the Fourier dimension of all three surfaces is $1$ in the following three subsections. This upper bound combined with Propositions \ref{p31} and \ref{p33} yields Theorem \ref{t2}. 

\subsection{Parabolic cone $C_{1}$}

\label{ex21}

We will show an upper bound on $\dim_F(C_{1})$, where the cone $C_{1}$ and the associated coordinate maps are defined in Definition \ref{d21}.

\begin{proposition}
\label{prop24}
    $\dim_F(C_{1}) \leq 1$. 
\end{proposition}
   
    \begin{proof}

  Let $\mu_0 \in \mathcal{M}(C_{1})$.  We need to show that if $\sup_{\xi \in \RR^{d+1}}|\xi|^{\frac{s}{2}}|\widehat{\mu_0}(\xi)| < \infty$, then $s \leq 1$. If $\text{spt } \mu_0 = \{ 0\}$, $\mu_0$ is the point mass at $0$, so $\widehat{\mu_0}(\xi)=1$, and $s=0$. Otherwise,  we assume $\mu_0(\{\Phi_{1}(x, h) | x \in \RR, h > 0\}) > 0$. There exist $R > 0$ and $0 < a < b$, such that for the set 
  \begin{equation}
  \label{eq250}
      C'_{1} : = \{\Phi_{1}(x, h) | x \in [-R, R], h \in [a, b] \},
  \end{equation}
  we have $\mu_0(C'_{1}) > 0$. Then we apply Lemma \ref{lem23} to obtain a measure $\mu \in \mathcal{M}(C'_{1})$ with $|\widehat{\mu}(\xi)| \leq c_{\mu} |\xi|^{-\frac{s}{2}}$ for a $c_{\mu} > 0$. We may further assume $\mu$ is a probability measure ($\norm{\mu}_{1} = 1$).
 
 For $t \in \RR$, we define $T_t: \RR \to \RR$, a shift on $\RR$, by
 \begin{equation}
 \label{eq163}
     T_t(x) : = x-t.
 \end{equation}
 
Each $T_t$ yields a new re-parametrization of $C_{1}$: 
\begin{equation*}
    \begin{split}
        C_{1} = & \{h(x-t, (x-t)^2, 1) | x \in \RR, h \in \RR \} \\
        = &  \{\Phi_{1}(T_t(x), h) | x \in \RR, h \in \RR \}.
    \end{split}
\end{equation*}

\begin{figure}[!h]
\centering
\beginpgfgraphicnamed{Illustration}
\begin{tikzpicture}
\draw[gray, thick] (0,1.5) -- (0,-2);
\draw[gray, thick] (2,1) -- (1, -2);
\filldraw[black] (1.9,1.05) circle (2pt) node[anchor=west]{$h(x, x^2, 1)$};
\draw (2,2) .. controls (-0.66, 1.5) .. (2,1);

\draw[gray, thick] (10,1.5) -- (10,-2);
\draw[gray, thick] (12,1) -- (11, -2);
\filldraw[black] (10, 1.5) circle (2pt) node[anchor=east]{$h(x-t, (x-t)^2, 1) $};
\draw (12,2) .. controls (9.34, 1.5) .. (12,1);
\draw[->]  (4, 0) -- (8, 0)
node [above,text width=3cm,text centered,midway]
{Re-parametrize by $T_t$
};

\draw[black, fill=gray, fill opacity=0.3] (2,1) -- (1.75, 0.45) -- (1.5,0.5) -- (1.62,1) -- cycle;

\draw[black, fill=black, fill opacity=0.6] (10.15, 0.9) -- (10, 1) -- (10, 1.5) -- (10.15, 1.35) -- cycle;
\end{tikzpicture}
\endpgfgraphicnamed
\caption{Illustration of the re-parametrization by $T_t$ on $C_{1}$}
\label{fig1}
\end{figure}
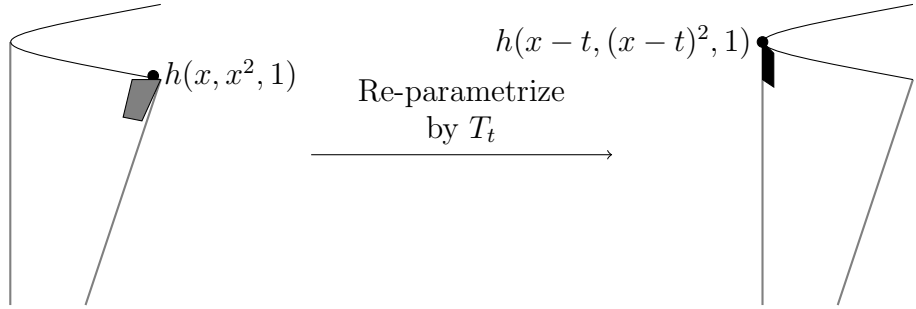

Given $\mu \in \mathcal{M}(C'_{1})$ and $t \in \RR$, we define $\nu_t \in \mathcal{M}(C_{1})$ using the re-parametrization as follows: for non-negative Borel functions $f$ on $\RR^3$,
\begin{equation*}
    \langle f, \nu_t \rangle : = \int f(\Phi_{1}(T_t(x), h)) d\mu \circ \Phi_{1}(x, h).
\end{equation*}

Geometrically, $\nu_t$ is a translation (by $t$) of $\mu$ along the parabola obtained from the horizontal cross-section of $C_{1}$. For example, suppose the gray and black patches on the cones in Figure \ref{fig1} are separated by the parameter $t$. If $\mu$ is supported in the gray patch of the cone $C_{1}$ on the left, $\nu_t$ is supported in the black patch of the cone on the right. We note that
\begin{equation}
\label{eq210}
    \begin{split}
        \widehat{\nu_t}(\xi) & = \int e^{-2 \pi i \Phi_{1}(T_t(x), h) \cdot \xi} d\mu \circ \Phi_{1}(x, h) \\
        & = \int e^{-2 \pi i h(x-t, (x-t)^2, 1) \cdot \xi} d\mu \circ \Phi_{1}(x, h) \\
        & = \int e^{-2 \pi i [h\xi_1(x-t)+h\xi_2(x-t)^2+h\xi_3)]} d\mu \circ \Phi_{1}(x, h).
    \end{split}
\end{equation}
The phase function is quadratic in $(x-t)$, and we expect the decay of $\widehat{\nu_t} (\xi)$ is slow along $\xi \| e_2$, where $e_2 \in \RR^3$, $(e_2)_j = \delta_{2, j}$. Now we create a new measure $\nu \in \mathcal{M}(C_{1})$ such that for non-negative Borel functions $f$ on $\RR^{3}$, 
\begin{equation*}
    \langle f, \nu \rangle : = \int_{\RR} \langle f, \nu_t \rangle \psi(t) dt,
\end{equation*}
where $\psi \in C_{0}^{\infty}(\RR)$ with $\psi \geq 0$ and $\psi(t) = 1$ for $|t| \leq R$. $\nu$ is a weighted average of the measures $\nu_t$. Then
\begin{equation}
\label{eq227}
    \widehat{\nu}(\xi) = \int_{\RR} \widehat{\nu_t}(\xi) \psi(t) dt.
\end{equation}

We have two facts about $\widehat{\nu}(\xi)$ when $\xi \| e_2$.

\begin{itemize}
    \item 
    Let $\eta_1: \RR \to \RR^3$, $\eta_1(t) : = (-2t, 1, t^2)$. Since we have a relation
\begin{equation}
\label{eq5}
    (x-t)^2 = (x, x^2, 1) \cdot \eta_1(t),
\end{equation}
and $h(x, x^2, 1) = \Phi_1(x, h)$ for $x, h \in \RR$, the above yields
\begin{equation}
\label{eq240}
    \begin{aligned}
        h(x-t)^2 = \Phi_1(x, h) \cdot \eta_1(t).
    \end{aligned}
\end{equation}
For $k>0$, 
\begin{equation}
\label{eq9}
    \begin{aligned}
\widehat{\nu_t}(ke_2) &  = \int e^{-2 \pi i hk(x-t)^2} d\mu \circ \Phi_{1}(x, h) & \text{from (\ref{eq210})}\\
 &  = \int e^{-2 \pi i k\Phi_1(x,h) \cdot \eta_1(t)} d\mu \circ \Phi_{1}(x, h) & \text{by (\ref{eq240})}\\
  &  = \int e^{-2 \pi i z \cdot k\eta_1(t)} d\mu(z)\\
& = \widehat{\mu}(k\eta_1(t)). 
\end{aligned}
\end{equation}
Since $|\widehat{\mu}(\xi)| \leq c_{\mu} |\xi|^{-\frac{s}{2}}$, by (\ref{eq227}),
\begin{equation}
\label{eq10}
    |\widehat{\nu}(ke_2)| = \left|\int_{\RR} \widehat{\mu}(k\eta_1(t)) \psi(t)dt\right| \leq c_{\mu} \norm{\psi}_{L^1} k^{-\frac{s}{2}}.
\end{equation}
\item On the other hand, by changing the order of integrations,
\begin{equation*}
\begin{aligned}
    \widehat{\nu}(ke_2) = & \int \int_{\RR} e^{-2 \pi i k h (x-t)^2}   \psi(t) dt d\mu \circ \Phi_{1}(x,h) & \text{from (\ref{eq210}), (\ref{eq227})}\\
    = &  \int \conj{I(2 \pi k h; F_x, \psi)} d\mu \circ \Phi_{1}(x,h) ,
\end{aligned}
\end{equation*}
where $I$ is defined in (\ref{eq724}) of Theorem \ref{t61}.b, and $F_x(t) = (x-t)^2$. By applying stationary phase (Theorem \ref{t61}.b) to $I$, if $kh>\lambda_0$ for $\lambda_0 \in \RR$ sufficiently large,
\begin{equation*}
    I(2 \pi k h; F_x, \psi) =  \left(-2ikh\right)^{-\frac{1}{2}}\psi(x) + E(x)(kh)^{-1},
\end{equation*}
where $E: [-R, R] \to \CC$ is a complex valued-function. In addition, $|E(x)| \leq E_0$ for all  $x\in[-R, R]$, where $E_0> 0$ depends on $\lambda_0$, $R$, and $\norm{\psi}_{C^4}$. Therefore, we conjugate and integrate in $d\mu \circ \Phi_{1}(x,h)$ to obtain
\begin{equation*}
\begin{split}
     & \int \conj{I(2 \pi k h; F_x, \psi)} d\mu \circ \Phi_{1}(x,h) \\
     = & (2i)^{-\frac{1}{2}} k^{-\frac{1}{2}}\int  h^{-\frac{1}{2}} \psi(x) d\mu \circ \Phi_{1}(x,h)+ k^{-1}\int \conj{E(x)}h^{-1} d\mu \circ \Phi_{1}(x,h ), 
\end{split}
\end{equation*}
and
\begin{equation*}
        \left| \widehat{\nu}(ke_2) - (2i)^{-\frac{1}{2}} k^{-\frac{1}{2}}\int  h^{-\frac{1}{2}} \psi(x) d\mu \circ \Phi_{1}(x,h) \right| \leq E_0 k^{-1}\int h^{-1} d\mu \circ \Phi_{1}(x,h ).
\end{equation*}
Since $0 < a \leq h \leq b$, $\psi(x)=1$ for $x \in [-R, R]$, and $\mu$ is a probability measure supported on $C'_1$ (defined in (\ref{eq250})), $$\int h^{-\frac{1}{2}} \psi(x) d\mu \circ \Phi_{1}(x,h)  \geq b^{-\frac{1}{2}},$$ and $$\int h^{-1} d\mu \circ \Phi_{1}(x,h) \leq a^{-1}.$$ So $|\widehat{\nu}(ke_2)| \geq \frac{1}{2}b^{-\frac{1}{2}}  k^{-\frac{1}{2}}$ for $k \geq k_0$, where $k_0>0$ depends on $E_0$, $a$, and $b$.
\end{itemize}
From these two facts, we have $s \leq 1$.
    \end{proof}
     
\subsection{A perturbed parabolic cone $C_{2}$}
\label{sec22}

The following proposition gives an upper bound on $\dim_F(C_{2})$, where the cone $C_2$ and the associated coordinate maps are defined in Definition \ref{d21}.

\begin{proposition}
\label{prop25}
    $\dim_F(C_{2}) \leq 1$.
\end{proposition}

\begin{remark}
     If we still use $T_t(x) = x-t$ as in (\ref{eq163}) and follow the procedure for $C_{1}$, we can no longer obtain (\ref{eq9}) and (\ref{eq10}) that relate the Fourier transform of the measures $\mu$, $\nu_t$, and $\nu$ on $C_{2}$. The issue is that no analog of (\ref{eq5}) holds in this case. Specifically, we ask readers to check that for $t \neq 0$, there does not exist $v_t \in \RR^3$, such that for all $x$ in a neighborhood of $0$,
$$(x-t)^2 = (x+x^3, x^2, 1) \cdot v_t.$$
\end{remark}

\begin{proof}

A preliminary task is to find re-parametrizations indexed by $t \in I_{2}$ of the curve $S_{2}$ given by $x \to T_t(x)$ such that there exists $\eta_2: I_{2} \to \RR^3$,
\begin{equation}
\label{re5}
    (T_t(x))^2 = (x+x^3, x^2, 1) \cdot \eta_2(t).
\end{equation}
Since for $h \in \RR$, $x \in I_2$, $\Phi_2(x, h) = h(x+x^3, x^2, 1)$, the above yields the following relation:
\begin{equation}
\label{re293}
    h (T_t(x))^2 = \Phi_2(x, h) \cdot \eta_2(t).
\end{equation}

Assuming for the moment that such $\{T_t\}$ exists, the remainder of the proof is completed as follows. We need to show that if $\mu \in \mathcal{M}(C_{2})$ and $|\widehat{\mu}(\xi)| \leq c_{\mu} |\xi|^{-\frac{s}{2}}$ for a $c_{\mu}>0$, then $s \leq 1$. We can reduce to the case where $\mu$ is supported in 
\begin{equation}
\label{eq409}
    C'_{2} : = \{\Phi_{2}(x, h) | x \in 2^{-1}I_{2}, h \in [a, b] \},
\end{equation}
for $0 < a < b$, and $\mu$ is a probability measure as in the proof of Proposition \ref{prop24}. 

Each $T_t$ yields a new re-parametrization of $C_{2}$: 
\begin{equation*}
    \begin{split}
        C_{2} \supset & \{h(\gamma_2(T_t(x)), (T_t(x))^2, 1) | x \in 2^{-1}I_2, h \in \RR \} \\
        = &  \{\Phi_{2}(T_t(x), h) | x \in 2^{-1}I_2, h \in \RR \}.
    \end{split}
\end{equation*}

Given $\mu \in \mathcal{M}(C'_{2})$ and $t \in I_{2}$, we use this re-parametrization to define $\nu_t \in \mathcal{M}(C_{2})$ as follows: for non-negative Borel functions $f$ on $\RR^3$,
\begin{equation}
\label{eq320}
    \langle f, \nu_t \rangle : = \int f(\Phi_{2}(T_t(x), h)) d\mu \circ \Phi_{2}(x, h).
\end{equation}

\begin{figure}[!h]
\centering
\beginpgfgraphicnamed{Illustration}
\begin{tikzpicture}
\draw[gray, thick] (0,1.5) -- (0,-2);
\draw[gray, thick] (2,1) -- (1, -2);
\filldraw[black] (1.9,1.05) circle (2pt) node[anchor=south west]{$h(x+x^3, x^2, 1)$};
\draw (2,2) .. controls (-0.66, 1.5) .. (2,1);

\draw[gray, thick] (10,1.5) -- (10,-2);
\draw[gray, thick] (12,1) -- (11, -2);
\filldraw[black] (10, 1.5) circle (2pt) node[anchor=east]{$h(x_t+x_t^3, x_t^2, 1) $};
\draw (12,2) .. controls (9.34, 1.5) .. (12,1);
\draw[->]  (4, 0) -- (8, 0)
node [above,text width=3cm,text centered,midway]
{Re-parametrize by $T_t$
};

\draw[black, fill=gray, fill opacity=0.3] (2,1) -- (1.75, 0.45) -- (1.5,0.5) -- (1.62,1) -- cycle;

\draw[black, fill=black, fill opacity=0.6] (10.15, 0.9) -- (10, 1) -- (10, 1.5) -- (10.15, 1.35) -- cycle;
\end{tikzpicture}
\endpgfgraphicnamed
\caption{Illustration of the shift $\Theta_t$ on $C_{2}$, where $x_t=T_t(x)$.}
\label{fig2}
\end{figure}
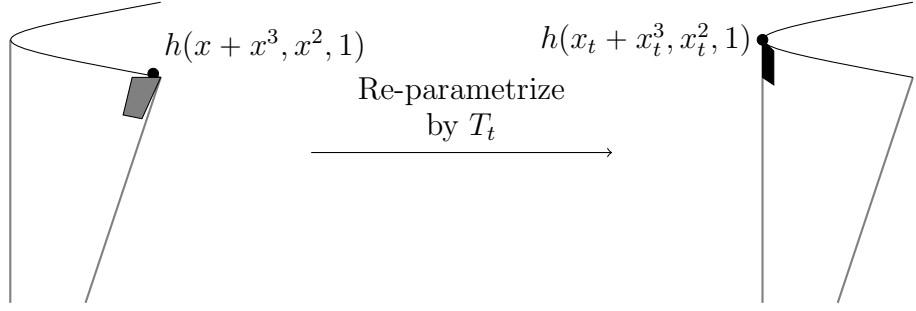

Geometrically, $\nu_t$ is a perturbed translation (by $t$) of $\mu$ along the curve $y=x+x^3$ obtained from the horizontal cross-section of $C_{2}$.

Now we create a new measure $\nu \in \mathcal{M}(C_{2})$ such that for non-negative Borel functions $f$ on $\RR^{3}$, 
\begin{equation*}
    \langle f, \nu \rangle : = \int_{\RR} \langle f, \nu_t\rangle \psi(t) dt,
\end{equation*}
where $\psi \in C_{0}^{\infty}(\RR)$, $\psi \geq 0$, $\text{spt }\psi \subset I_{2}$, and  $\psi(t) = 1$ for $t \in 2^{-1}I_{2}$. $\nu$ is a weighted average of the measures $\nu_t$. Then
\begin{equation}
\label{eq369}
    \widehat{\nu}(\xi) = \int_{\RR} \widehat{\nu_t}(\xi) \psi(t) dt.
\end{equation}

We have two facts about $\widehat{\nu}(\xi)$ when $\xi \| e_2$.

\begin{itemize}
    \item 
    For $k>0$, 
\begin{equation*}
\begin{aligned}
\widehat{\nu_t}(ke_2) &  = \int e^{-2 \pi i \Phi_2(T_t(x), h) \cdot ke_2} d\mu \circ \Phi_2(x, h) & \text{from (\ref{eq320})}\\
&  = \int e^{-2 \pi i k h T_t(x)^2 } d\mu \circ \Phi_2(x, h) & \text{by (\ref{eq117})}\\
 &  = \int e^{-2 \pi i \Phi_2(x, h) \cdot k\eta_2(t)} d\mu \circ \Phi_2(x, h) & \text{by (\ref{re293})}\\
 &  = \int e^{-2 \pi i z \cdot k\eta_2(t)} d\mu(z)\\
& = \widehat{\mu}(k\eta_2(t)). 
\end{aligned}
\end{equation*} If $|\eta_2(t)| \geq 1$ and $|\widehat{\mu}(\xi)| \leq c_{\mu} |\xi|^{-\frac{s}{2}}$, by (\ref{eq369}),
\begin{equation*}
    |\widehat{\nu}(ke_2)| = \left|\int_{\RR} \widehat{\mu}(k\eta_2(t)) \psi(t)dt\right| \leq c_{\mu} \norm{\psi}_{L^1} k^{-\frac{s}{2}}.
\end{equation*}
\item On the other hand, by changing the order of integrations,
\begin{equation*}
\begin{aligned}
    \widehat{\nu}(ke_2) = & \int \int_{\RR} e^{-2 \pi i k h T_t(x)^2}   \psi(t) dt d\mu \circ \Phi_2(x,h) & \text{from (\ref{eq320}), (\ref{eq369})} \\
    = & \int \conj{I(2\pi k h; F_x,\psi)}d\mu \circ \Phi_2(x,h),
\end{aligned}
\end{equation*}
where $F_x(t) = T_t(x)^2$, and $I$ is defined in (\ref{eq724}) of Theorem \ref{t61}. Suppose $F_x$ satisfies the conditions stated in Lemma \ref{lem26}. We can estimate the integrand $\conj{I(2\pi k h; F_x,\psi)}$ using stationary phase (Theorem \ref{t61}.b) to obtain
\begin{equation*}
    I(2\pi k h; F_x,\psi) =  \left[-ikh\frac{\partial^2 F_x}{\partial t^2}(x)\right]^{-\frac{1}{2}}\psi(x)+ E(x)(kh)^{-1},
\end{equation*}
where  $kh>\lambda_0$ for $\lambda_0 \in \RR$ sufficiently large and $E: I_2 \to \CC$ is a complex-valued function. In addition, $|E(x)| \leq E_0$, where $E_0>0$ depends on $\lambda_0$, size of $I_2$, $\sup \{\norm{F_x}_{C_7}| x \in 2^{-1} I_2 \}$, and $\norm{\psi}_{C_4}$. Therefore, we conjugate and integrate in $d\mu \circ \Phi_2(x,h)$ to obtain
\begin{equation*}
\begin{aligned}
    & \int \conj{I(\pi k h; F_x,\psi)}d\mu \circ \Phi_2(x,h) \\
    = & k^{-\frac{1}{2}} \int  \left[ih\frac{\partial^2 F_x}{\partial t^2}(x)\right]^{-\frac{1}{2}} \psi(x) d\mu \circ \Phi_2(x,h) + k^{-1}\int \conj{E(x)}h^{-1} d\mu \circ \Phi_2(x,h),
\end{aligned}
\end{equation*} and
\begin{equation*}
\left| \widehat{\nu}(ke_2) - k^{-\frac{1}{2}} \int  \left[ih\frac{\partial^2 F_x}{\partial t^2}(x)\right]^{-\frac{1}{2}} \psi(x) d\mu \circ \Phi_2(x,h) \right| \leq E_0 k^{-1}\int  h^{-1} d\mu \circ \Phi_2(x,h).
\end{equation*} 
Suppose $\epsilon \leq \frac{\partial^2 F_x}{\partial t^2}(x) \leq \epsilon^{-1}$ for an $\epsilon > 0$. Since $0 < a \leq h \leq b$, $\psi(x) = 1$ for $x \in 2^{-1}I_2$, $\mu$ is a probability measure supported on $C'_2$ (defined in (\ref{eq409})),  
\begin{equation*}
     \int  \left[h\frac{\partial^2 F_x}{\partial t^2}(x)\right]^{-\frac{1}{2}} \psi(x)d\mu \circ \Phi_2(x,h) \geq \epsilon^{-\frac{1}{2}} b^{-\frac{1}{2}},
\end{equation*}
and
\begin{equation*}
     \int h^{-1} d\mu \circ \Phi_2(x,h) \leq a^{-1}.
\end{equation*}
So $|\widehat{\nu}(ke_2)| \geq \frac{1}{2} (\epsilon b)^{-\frac{1}{2}}  k^{-\frac{1}{2}}$ for $k \geq k_0$, where $k_0>0$ depends on $\epsilon, E_0, a$ and $b$.

\end{itemize}
From these two facts, we have $s \leq 1$. It remains to show that such $\{T_t\}$ exists in the following lemma.
\end{proof}

\begin{lemma}
\label{lem26}
\begin{enumerate}
    \item Let $I_{2}$ be a neighborhood of $0$ that is sufficiently small. Then for each $t \in I_{2}$, there exist $T_t: I_{2} \to \RR$, $\eta_2: I_{2} \to \RR^3$, $|\eta_2(t)| \geq 1$ that satisfy \begin{equation}
    (T_t(x))^2 = (x+x^3, x^2, 1) \cdot \eta_2(t). \tag{\ref{re5} revisited}
\end{equation}
    \item For each $x$, let $F_x: I_{2} \to \RR$, $F_x(t) : = T_t(x)^2$. $F_x$ satisfies the following three properties:
\begin{enumerate}
    \item $F_x(x)=0$.
    \item For each $x$, $\frac{\partial F_x}{\partial t}(t)= 0$ and $t \in I_{2} $ if and only if $t=x$.				
    \item $\frac{\partial^2 F_x}{\partial t^2}(0) = 2$, and there exists $ > 0$, such that for each $x \in I_{2}$, $\epsilon \leq \frac{\partial^2 F_x}{\partial t^2}(t) \leq \epsilon^{-1}$.
\end{enumerate}
\end{enumerate}

\end{lemma}

\begin{proof}
We first solve for $\eta_2$. Suppose $F: I_{2} \times I_{2} \to \RR$ is $$F(x, t) = F_x(t) = (T_t(x))^2 = (x+x^3, x^2, 1) \cdot \eta_2(t).$$ Then given $F(x, x)=0$ for $x \in I_{2}$, $\frac{\partial F}{\partial t}(x, x)= 0$ is equivalent to $\frac{\partial F}{\partial x}(x, x)= 0$.  One solution is 
\begin{equation*}
    \eta_2(t) : = \left( -\frac{2t}{1+3t^2}, 1, \frac{2t(t+t^3)}{1+3t^2}-t^2 \right).
\end{equation*}
This choice of $\eta_2$ satisfies the requirement that $|\eta_2(t)| \geq 1$. We also note that the function $F_x$ satisfies all three properties in the list by the continuity of $F$ and the implicit function theorem.

We can return to define the transformation $T_t$. Since 
\begin{equation*}
\begin{aligned}
    (x+x^3, x^2, 1) \cdot \eta_2(t) & = \frac{-2t(x+x^3)+x^2(1+3t^2)+t^2-t^4}{1+3t^2} \\
    & = (x-t)^2\frac{1-3x^2+4x(x-t)-(x-t)^2}{1+3t^2},
\end{aligned}
\end{equation*}
we can define
$$T_t(x) : = (x-t) \sqrt{\frac{1-3x^2+4x(x-t)-(x-t)^2}{1+3t^2}}.$$
\end{proof}

\begin{remark}
    We note that $T_t(t)=0$, $T_0(x)=x$ for $x \in I_2$, and $$\sqrt{\frac{1-3x^2+4x(x-t)-(x-t)^2}{1+3t^2}} \approx 1$$ if $x, t$ are close to $0$. So the transformations $\{T_t\}$ in this case are perturbed variants of the ones used in $C_1$.
\end{remark}

\subsection{Parabolic cylinder $D_{1}$}

\label{ex23}

We will show an upper bound on $\dim_F(D_{1})$, where the cylinder $D_{1}$ and the associated coordinate maps are defined in Definition \ref{d21}.

\begin{proposition}
\label{prop632}
    $\dim_F(D_{1}) \leq 1$. 
\end{proposition}

\begin{remark}
    The idea of creating an average measure from an old one still solves this case, even though the computations are slightly different compared to Proposition \ref{prop24} for the cone $C_1$.
\end{remark}

\begin{proof}

  Let $\mu \in \mathcal{M}(D_{1})$.  We need to show that if $\mu \in \mathcal{M}(D_{1})$ and $|\widehat{\mu}(\xi)| \leq c_{\mu} |\xi|^{-\frac{s}{2}}$ for a $c_{\mu}>0$, then $s \leq 1$. We can reduce to the case where $\mu$ is supported in 
\begin{equation}
\label{eq640}
    D'_{1} : = \{\widetilde{\Phi}_{1}(x, h) | x \in [-R, R], h \in \RR \}
\end{equation}
and $\mu$ is a probability measure as in the proof of Proposition \ref{prop24}. 
 
 For $t \in \RR$, we define $T_t: \RR \to \RR$, a shift on $\RR$, by
 \begin{equation}
 \label{eq647}
     T_t(x) : = x-t.
 \end{equation}
 
Each $T_t$ yields a new re-parametrization of $D_{1}$: 
\begin{equation*}
    \begin{split}
        D_{1} = & \{(x-t, (x-t)^2, h) | x \in \RR, h \in \RR \} \\
        = &  \{\widetilde{\Phi}_{1}(T_t(x), h) | x \in \RR, h \in \RR \}.
    \end{split}
\end{equation*}

\begin{figure}[!h]
\centering
\beginpgfgraphicnamed{Illustration}
\begin{tikzpicture}
\draw[gray, thick] (0,1.5) -- (0,-2);
\draw[gray, thick] (2,1) -- (2, -2);
\filldraw[black] (1.9,1.05) circle (2pt) node[anchor=west]{$(x, x^2, h)$};
\draw (2,2) .. controls (-0.66, 1.5) .. (2,1);

\draw[gray, thick] (10,1.5) -- (10,-2);
\draw[gray, thick] (12,1) -- (12, -2);
\filldraw[black] (10, 1.5) circle (2pt) node[anchor=east]{$(x-t, (x-t)^2, h) $};
\draw (12,2) .. controls (9.34, 1.5) .. (12,1);
\draw[->]  (4, 0) -- (8, 0)
node [above,text width=3cm,text centered,midway]
{Re-parametrize by $T_t$
};

\draw[black, fill=gray, fill opacity=0.3] (2,1) -- (2, 0.45) -- (1.55,0.5) -- (1.55,1) -- cycle;

\draw[black, fill=black, fill opacity=0.6] (10.15, 0.9) -- (10, 1) -- (10, 1.5) -- (10.15, 1.35) -- cycle;
\end{tikzpicture}
\endpgfgraphicnamed
\caption{Illustration of the re-parametrization by $T_t$ on $D_{1}$}
\label{fig3}
\end{figure}
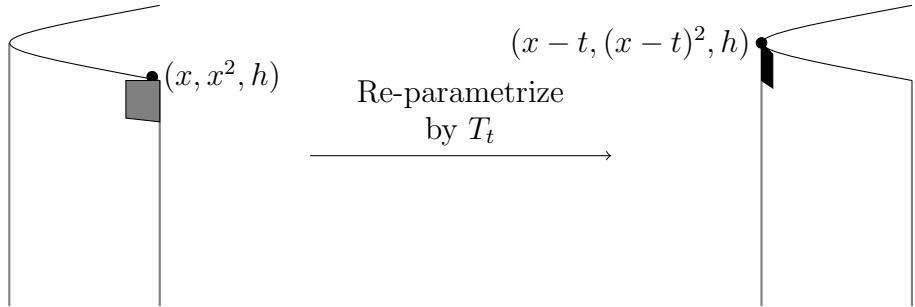

Given $\mu \in \mathcal{M}(D_1)$ and $t \in \RR$, we define $\nu_t \in \mathcal{M}(D_{1})$ using the re-parametrization as follows: for non-negative Borel functions $f$ on $\RR^3$,
\begin{equation*}
    \langle f, \nu_t \rangle : = \int f(\widetilde{\Phi}_{1}(T_t(x), h)) d\mu \circ \widetilde{\Phi}_{1}(x, h).
\end{equation*}

Geometrically, $\nu_t$ is a translation (by $t$) of $\mu$ along the parabola obtained from the horizontal cross-section of $D_{1}$. For example, suppose the gray and black patches on the cones in Figure \ref{fig3} are separated by the parameter $t$. If $\mu$ is supported in the gray patch of the cone $D_{1}$ on the left, $\nu_t$ is supported in the black patch of the cone on the right. We note that
\begin{equation}
\label{eq693}
    \begin{split}
        \widehat{\nu_t}(\xi) & = \int e^{-2 \pi i \widetilde{\Phi}_{1}(T_t(x), h) \cdot \xi} d\mu \circ \widetilde{\Phi}_{1}(x, h) \\
        & = \int e^{-2 \pi i (x-t, (x-t)^2, h) \cdot \xi} d\mu \circ \widetilde{\Phi}_{1}(x, h) \\
        & = \int e^{-2 \pi i [\xi_1(x-t)+\xi_2(x-t)^2+h\xi_3)]} d\mu \circ \widetilde{\Phi}_{1}(x, h).
    \end{split}
\end{equation}
The phase function is quadratic in $(x-t)$, and we expect the decay of $\widehat{\nu_t} (\xi)$ is slow along $\xi \| e_2$, where $e_2 \in \RR^3$, $(e_2)_j = \delta_{2, j}$. Now we create a new measure $\nu \in \mathcal{M}(D_{1})$ such that for non-negative Borel functions $f$ on $\RR^{3}$, 
\begin{equation*}
    \langle f, \nu \rangle : = \int_{\RR} \langle f, \nu_t \rangle \psi(t) dt,
\end{equation*}
where $\psi \in C_{0}^{\infty}(\RR)$ with $\psi \geq 0$ and $\psi(t) = 1$ for $|t| \leq R$. $\nu$ is a weighted average of the measures $\nu_t$. Then
\begin{equation}
\label{eq706}
    \widehat{\nu}(\xi) = \int_{\RR} \widehat{\nu_t}(\xi) \psi(t) dt.
\end{equation}

We have two facts about $\widehat{\nu}(\xi)$ when $\xi \| e_2$.

\begin{itemize}
    \item 
    Let $\widetilde{\eta_1}: \RR \to \RR^3$, $\widetilde{\eta_1}(t) = (-2t, 1, 0)$. We have a relation
\begin{equation}
\label{eq716}
    (x-t)^2 = (x, x^2, h) \cdot \widetilde{\eta_1}(t)+t^2 = \widetilde{\Phi_1}(x, h) \cdot \widetilde{\eta_1}(t)+t^2  .
\end{equation}
For $k>0$, 
\begin{equation*}
    \begin{aligned}
\widehat{\nu_t}(ke_2) &  = \int e^{-2 \pi i k(x-t)^2} d\mu \circ \widetilde{\Phi_{1}}(x, h) & \text{from (\ref{eq706})}\\
 &  = e^{-2\pi i k t^2}\int e^{-2 \pi i k \widetilde{\Phi_1}(x,h) \cdot \widetilde{\eta_1}(t)} d\mu \circ \widetilde{\Phi_{1}}(x, h) & \text{by (\ref{eq716})}\\
  &  = e^{-2\pi i k t^2} \int e^{-2 \pi i z \cdot k\widetilde{\eta_1}(t)} d\mu(z)\\
& = e^{-2\pi i k t^2}\widehat{\mu}(k\widetilde{\eta_1}(t)). 
\end{aligned}
\end{equation*}
Since $|\widehat{\mu}(\xi)| \leq c_{\mu} |\xi|^{-\frac{s}{2}}$, by (\ref{eq706}),
\begin{equation*}
    |\widehat{\nu}(ke_2)| = \left|\int_{\RR} e^{-2\pi i k t^2} \widehat{\mu}(k\widetilde{\eta_1}(t)) \psi(t)dt\right| \leq c_{\mu} \norm{\psi}_{L^1} k^{-\frac{s}{2}}.
\end{equation*}
\item On the other hand, by changing the order of integrations,
\begin{equation*}
\begin{aligned}
    \widehat{\nu}(ke_2) = & \int \int_{\RR} e^{-2 \pi i k (x-t)^2}   \psi(t) dt d\mu \circ \widetilde{\Phi_{1}}(x,h) & \text{from (\ref{eq693}), (\ref{eq706})}\\
    = &  \int \conj{I(2 \pi k; F_x, \psi)} d\mu \circ \widetilde{\Phi_{1}}(x,h) ,
\end{aligned}
\end{equation*}
where $I$ is defined in (\ref{eq724}) of Theorem \ref{t61}.b, and $F_x(t) = (x-t)^2$. By applying stationary phase (Theorem \ref{t61}.b) to $I$, if $k>\lambda_0$ for $\lambda_0 \in \RR$ sufficiently large,
\begin{equation*}
    I(2 \pi k; F_x, \psi) =  (-2ik)^{-\frac{1}{2}}\psi(x) + E(x)k^{-1},
\end{equation*}
where $E: [-R, R] \to \CC$ is a complex valued-function. In addition, $|E(x)| \leq E_0$ for all $x\in[-R, R]$, where $E_0>0$ depends on $\lambda_0$, $R$, and $\norm{\psi}_{C_4}$. Therefore, we conjugate and integrate in $d\mu \circ \Phi_{1}(x,h)$ to obtain
\begin{equation*}
\begin{split}
     & \int \conj{I(2 \pi k; F_x, \psi)} d\mu \circ \widetilde{\Phi}_{1}(x,h) \\
     = & (2i)^{-\frac{1}{2}} k^{-\frac{1}{2}}\int   \psi(x) d\mu \circ \widetilde{\Phi}_{1} (x,h)+ k^{-1}\int \conj{E(x)} d\mu \circ \widetilde{\Phi}_{1}(x,h ), 
\end{split}
\end{equation*}
and
\begin{equation*}
        \left| \widehat{\nu}(ke_2) - (2i)^{-\frac{1}{2}} k^{-\frac{1}{2}}\int   \psi(x) d\mu \circ\widetilde{\Phi}_{1}(x,h) \right| \leq E_0 k^{-1}\int d\mu \circ \widetilde{\Phi}_{1}(x,h ).
\end{equation*}
Since $\psi(x)=1$ for $x \in [-R, R]$, and $\mu$ is a probability measure supported on $D'_1$ (defined in (\ref{eq640})), $$\int  \psi(x) d\mu \circ \widetilde{\Phi}_{1}(x,h) = \int d\mu \circ \widetilde{\Phi}_{1}(x,h)(x,h)= 1.$$ 
Therefore, $|\widehat{\nu}(ke_2)| \geq \frac{1}{2}  k^{-\frac{1}{2}}$ for $k \geq k_0$, where $k_0>0$ depends on $E_0$.
\end{itemize}
From these two facts, we have $s \leq 1$.
    \end{proof}

\section{Proof of Proposition \ref{p32}}
\label{sec764}

Our goal in this section is to show that if $\mu_0 \in \mathcal{M}(C)$, and $\sup_{\xi \in \RR^{d+1}} |\xi|^{\frac{s}{2}}|\widehat{\mu_0}(\xi)| < \infty$, then $s \leq d-1$.  We first show this statement for a simple case. Let $$ C_{\partial S} : = \{h(x, 1) | x \in \partial S , h \in \RR\}.$$ The set $C_{\partial S} \cup \{0\}$ is at most $d-1$ dimensional. If $\text{spt} \mu_0 \subset C_{\partial S} \cup \{0\}$, $s \leq d-1$ by Frostman's lemma.

The case remaining is where $\text{spt } \mu_0 \not \subset C_{\partial S} \cup \{ 0\}$. There exist $p \in \text{Int}S$ and $h \neq 0$, such that $h(p, 1) \in \text{spt} \mu_0$. Without loss of generality, we assume $h>0$.

Next, we find a coordinate map near the point $h(p, 1)$ that is used throughout this section. 
For $f \in C^{\infty}(\RR^n)$, we denote ${\bf{H}} f$, the Hessian of $f$, as the $n$ by $n$ matrix
\begin{equation}
\label{eq796}
   {\bf H}f := \left( \frac{\partial^2 f}{\partial y_j \partial y_k} \right).
\end{equation}
For $F \in C^{\infty}(\RR^n \times \RR^n)$ with inputs $x$ and $y$, we use
 ${\bf H}_x F(x, y)$ with the subscript $x$ to denote ${\bf H} F_y(x)$, where $F_y \in C^{\infty}(\RR^{n})$, $F_y(x) = F(x, y)$.

By Lemma \ref{lem63}, after some rotation and translation that move $p$ to $0 \in \RR^{d}$, there exists a coordinate map $\phi_{0}: V  \to S$, where $V \subset \RR^{d-1}$ is open and contains $0$, and
$$\phi_{0}(y)^T := (y, g(y))$$ for a $g: V \to \RR$ with $g(0)=0$, $\nabla g(0)=0$, and  $\det {\bf H}g(0) \neq 0$. We note that the Fourier dimension of a set is rotation and translation invariant.

Let $Q_m: \RR^{n} \to \RR$ be the unit quadratic form 
\begin{equation}
\label{eq764}
  Q_m(y) : =\sum_{j=1}^{m} y_j^2 - \sum_{j=m+1}^{n} y_j^2  
\end{equation} where $0 \leq m \leq n$, and $m$ is fixed. By applying Morse's lemma (Lemma \ref{lem32}) to $g$, there exist $U \subset \RR^{d-1}$, a neighborhood of $0$, and a smooth $\varphi: U \to V$ diffeomorphic to its image, such that $\phi_{0}(V) \subset S$ can be re-parameterize by $\phi: U \to S$, where
\begin{equation}
\label{eq581}
    \phi(x)^T := \phi_{0}(\varphi(x))^T=(\varphi(x)^T, Q_m(x)),
\end{equation}
 for $0 \leq m \leq d-1$.

Then the cone $C \subset \RR^{d+1}$ generated by $S$ has a coordinate map $\Phi : U \times \RR  \to \RR^{d+1}$ given by
\begin{equation}
\label{eq18}
    \Phi(x, h): = h(\phi(x), 1).
\end{equation}

Let $c<1$ to be chosen later. There exist $0<a<b$ such that for
\begin{equation}
\label{eq599}
    C' := \{\Phi(x, h) | x \in 2^{-1}cU, h \in [a, b]\},
\end{equation}
$\mu_0(C') > 0$. We apply Lemma \ref{lem23} to obtain a measure $\mu \in \mathcal{M}(C')$ with $|\widehat{\mu}(\xi)| \leq c_{\mu} |\xi|^{-\frac{s}{2}}$. By a normalization, we assume $\mu$ is a probability measure. The rest of the proof generalizes the arguments in section \ref{sec22}. 

\subsection{Constructing an average measure}

A preliminary task is to find re-parametrizations indexed by $t \in cU$ of the surface $S$ given by $x \to T_t(x)$ such that there exists $\eta: cU \to \RR^{d+1}$,
\begin{equation}
\label{eq6}
    Q_m(T_t(x)) = ( \phi(x), 1 ) \cdot \eta(t).
\end{equation}
Since for $h \in \RR$, $x \in cU$, $\Phi(x, h) = h(\phi(x), 1)$, the above yields the following relation:
\begin{equation}
\label{eq614}
    h   Q_m(T_t(x)) = \Phi(x, h) \cdot \eta(t).
\end{equation}

Assuming for the moment that such $\{T_t\}$ exists, the remainder of the proof is completed as follows. Each $T_t$ yields a new re-parametrization of $C$: 
\begin{equation*}
    \begin{split}
        C \supset & \{h(\phi(T_t(x)), 1) | x \in cU, h \in \RR \} \\
        = &  \{\Phi(T_t(x), h) | x \in cU, h \in \RR \}.
    \end{split}
\end{equation*}

Given $\mu \in \mathcal{M}(C')$ and $t \in cU$, we use this re-parametrization to define $\nu_t \in \mathcal{M}(C)$ as follows: for non-negative Borel functions $f$ on $\RR^{d+1}$,
\begin{equation}
\label{eq628}
    \langle f, \nu_t \rangle := \int f(\Phi(T_t(x), h)) d\mu \circ \Phi(x, h).
\end{equation}
Now we create a new measure $\nu \in \mathcal{M}(C)$ such that for non-negative Borel functions $f$ on $\RR^{d+1}$, 
\begin{equation*}
    \langle f, \nu \rangle := \int_{\RR^{d-1}} \langle f, \nu_t\rangle \psi(t) dt,
\end{equation*}
where $\psi \in C_{0}^{\infty}(\RR^{d-1})$ with $\psi \geq 0$, $\text{spt} \psi \subset cU$, and $\psi(t) = 1$ for $t \in 2^{-1}cU$. $\nu$ is a weighted average of the measures $\nu_t$. Then
\begin{equation}
\label{eq637}
    \widehat{\nu}(\xi) = \int_{\RR^{d-1}} \widehat{\nu_t}(\xi) \psi(t) dt.
\end{equation}

We have two lemmas about $\widehat{\nu}(\xi)$ when $\xi \| e_d$, where $e_d \in \RR^{d+1}$, $(e_d)_k = \delta_{d, k}$.

\begin{lemma}
\label{lem53}
    If $|\eta(t)| \geq 1$, $|\widehat{\nu}(k e_d)| \leq c_{\mu} \norm{\psi}_{L^1} k^{-\frac{s}{2}}.$ for $k>0$.
\end{lemma}

\begin{proof}
     For $k>0$, 
\begin{equation*}
\begin{aligned}
\widehat{\nu_t}(ke_d) &  = \int e^{-2 \pi i \Phi(T_t(x), h) \cdot ke_d} d\mu \circ \Phi(x, h) & \text{from (\ref{eq628})}\\
&  = \int e^{-2 \pi i k h Q_m(T_t(x)) } d\mu \circ \Phi(x, h) & \text{by (\ref{eq581}), (\ref{eq18})}\\
 &  = \int e^{-2 \pi i \Phi(x, h) \cdot k \eta(t)} d\mu \circ \Phi(x, h) & \text{by (\ref{eq614})}\\
 &  = \int e^{-2 \pi i z \cdot k\eta(t)} d\mu(z)\\
& = \widehat{\mu}(k\eta(t)). 
\end{aligned}
\end{equation*} Then by (\ref{eq637})
\begin{equation*}
    |\widehat{\nu}(ke_d)| = \left|\int_{\RR^{d-1}} \widehat{\mu}(k\eta(t)) \psi(t)dt\right| \leq c_{\mu} \norm{\psi}_{L^1} k^{-\frac{s}{2}}.
\end{equation*}
\end{proof}

\begin{lemma}
\label{lem54}
Assuming Lemma \ref{lem708}, there exist $c_0, k_0>0$ such that for $k \geq  k_0$, $|\widehat{\nu}(ke_d)| \geq c_0 k^{-\frac{d-1}{2}}$.
\end{lemma}

\begin{proof}   
By changing the order of integrations,
\begin{equation*}
\begin{aligned}
    \widehat{\nu}(ke_d) = & \int \int e^{-2 \pi i k h Q_m(T_t(x)) } d\mu \circ \Phi(x, h) & \text{from (\ref{eq628}), (\ref{eq637})} \\
    = & \int \conj{I(2\pi k h; F_x,\psi)}d\mu \circ \Phi(x,h),
\end{aligned}
\end{equation*}
where $F_x(t) =  Q_m(T_t(x))$, and $I$ is defined in (\ref{eq724}) of Theorem \ref{t61}. Suppose $F_x$ satisfies the conditions stated in Lemma \ref{lem708}. Let $c_{m}^{(d)} = i^{\frac{3(d-1)}{2}-m}$, and $H(x) = \det {\bf H} F_x(x)$, where the Hessian ${\bf H}$ is defined in (\ref{eq796}). We can estimate the integrand $\conj{I(2\pi k h; F_x,\psi)}$ using stationary phase (Theorem \ref{t61}) to obtain
\begin{equation*}
    I(2\pi k h; F_x,\psi) =  c^{(d)}_m (k h)^{-\frac{d-1}{2}} |H(x)|^{-\frac{1}{2}}\psi(x) + E(x)(kh)^{-\frac{d}{2}},
\end{equation*}
where $kh>\lambda_0$ for $\lambda_0 \in \RR$ sufficiently large and $E: cU \to \CC$ is a complex-valued function. In addition, $|E(x)| \leq E_0$, where $E_0>0$ depends on $\lambda_0$, size of $cU$, $\sup \{\norm{F_x}_{C^{d+5}} | x \in cU\}$, and $\norm{\psi}_{C^{d+2}}$. Therefore, we conjugate and integrate in $d\mu \circ \Phi(x,h)$ to obtain
\begin{equation*}
\begin{aligned}
    & \int \conj{I(2\pi k h; F_x,\psi)}d\mu \circ \Phi(x,h) \\
    = & \conj{c^{(d)}_m}k^{-\frac{d-1}{2}}\int  h^{-\frac{d-1}{2}} |H(x)|^{-\frac{1}{2}} \psi(x) d\mu \circ \Phi(x,h) + k^{-\frac{d}{2}}\int \conj{E(x)}h^{-\frac{d}{2}} d\mu \circ \Phi(x,h),
\end{aligned}
\end{equation*} and
\begin{equation*}
\left| \widehat{\nu}(ke_d) - \conj{c^{(d)}_m} k^{-\frac{d-1}{2}}\int h^{-\frac{d-1}{2}} |H(x)|^{-\frac{1}{2}} \psi(x)  d\mu \circ \Phi(x,h)  \right| \leq  E_0 k^{-\frac{d}{2}}\int h^{-\frac{d}{2}} d\mu \circ \Phi(x,h).
\end{equation*} 
Suppose $\epsilon \leq  |H(x)| \leq \epsilon^{-1}$ for an $\epsilon > 0$. Since $0 < a \leq h \leq b$, $\psi(x) = 1$ for $x \in 2^{-1}cU$, $\mu$ is a probability measure supported on $C'$ ($C'$ is defined in (\ref{eq599})), 
\begin{equation*}
    \int  h^{-\frac{d-1}{2}} |H(x)|^{-\frac{1}{2}} \psi(x) d\mu \circ \Phi(x,h) \geq \epsilon^{-\frac{1}{2}} b^{-\frac{d-1}{2}},
\end{equation*}
and
\begin{equation*}
     \int h^{-\frac{d}{2}} d\mu \circ \Phi(x,h) \leq a^{-\frac{d}{2}}.
\end{equation*}
So $|\widehat{\nu}(ke_d)| \geq c_0 k^{-\frac{d-1}{2}}$ for $k \geq k_0$, where $c_0, k_0>0$ depend on $\epsilon, E_0, a$, $b$, and $d$.

\end{proof}

From these two lemmas, we have $s \leq d-1$. It remains to show that such $\{T_t\}$ exists in Section \ref{sec_ex}.

\section{Proof of Proposition \ref{p34}}
\label{sec981}

Our goal in this section is to show that if $\mu_0 \in \mathcal{M}(D)$, and $\sup_{\xi \in \RR^{d+1}}|\xi|^{\frac{s}{2}} |\widehat{\mu_0}(\xi)| < \infty$, then $s \leq d-1$.  We first show this statement for a simple case. Let $$D_{\partial S} := \partial S \times \RR \{(x, h) | x \in \partial S , h \in \RR\}.$$ The set $D_{\partial S}$ is at most $d-1$ dimensional. If $\text{spt} \mu_0 \subset D_{\partial S}$, $s \leq d-1$ by Frostman's lemma. The case remaining is where $\text{spt } \mu_0 \not \subset D_{\partial S}$. There exist $p \in \text{Int}S$ and $h \in \RR$, such that $(p, h) \in \text{spt} \mu_0$. 

Next, we find a coordinate map near the point $(p, h)$ that is used throughout this section. By Lemma \ref{lem63}, after some rotation and translation that move $p$ to $0 \in \RR^{d}$, there exists a coordinate map $\phi_{0}: V  \to S$, where $V \subset \RR^{d-1}$ is open and contains $0$, and
$$\phi_{0}(y)^T := (y, g(y))$$ for a $g: V \to \RR$ with $g(0)=0$, $\nabla g(0)=0$, and  $\det {\bf H}g(0) \neq 0$, where the Hessian ${\bf H}$ is defined in (\ref{eq796}). We note that the Fourier dimension of a set is rotation and translation invariant.

By applying Morse's lemma (Lemma \ref{lem32}) to $g$, there exist $U \subset \RR^{d-1}$, a neighborhood of $0$, and a smooth $\varphi: U \to V$ diffeomorphic to its image, such that $\phi_{0}(V) \subset S$ can be re-parameterize by $\phi: U \to S$, where
\begin{equation}
\label{eq991}
    \phi(x)^T := \phi_{0}(\varphi(x))^T=(\varphi(x)^T, Q_m(x)),
\end{equation}
 for $0 \leq m \leq d-1$, and the quadratic form $Q_m$ is defined in (\ref{eq764}).

Then the cylinder $D \subset \RR^{d+1}$ generated by $S$ has a coordinate map $\widetilde{\Phi} : U \times \RR  \to \RR^{d+1}$ given by
\begin{equation}
\label{eq998}
    \widetilde{\Phi}(x, h) := (\phi(x), h).
\end{equation}

Let $c<1$ to be chosen later, and
\begin{equation}
\label{eq1004}
    D' := \{\widetilde{\Phi}(x, h) | x \in 2^{-1}cU, h \in \RR\},
\end{equation}
$\mu_0(D') > 0$. We apply Lemma \ref{lem23} to obtain a measure $\mu \in \mathcal{M}(D')$ with $|\widehat{\mu}(\xi)| \leq c_{\mu} |\xi|^{-\frac{s}{2}}$. By a normalization, we assume $\mu$ is a probability measure. The rest of the proof adapts the one in section \ref{sec764} for the cone $C$.

\subsection{Constructing an average measure}

A preliminary task is to find re-parametrizations indexed by $t \in cU$ of the surface $S$ given by $x \to T_t(x)$ such that there exist $\widetilde{\eta}: cU \to \RR^{d+1}$ and $\widetilde{\rho}: cU \to \RR$ with 
\begin{equation}
\label{eq1013}
    Q_m(T_t(x)) = \widetilde{\Phi}(x, h) \cdot \widetilde{\eta}(t) + \widetilde{\rho}(t).
\end{equation}

Assuming for the moment that such $\{T_t\}$ exists, the remainder of the proof is completed as follows. Each $T_t$ yields a new re-parametrization of $D$: 
\begin{equation*}
    \begin{split}
        D \supset &  \{\widetilde{\Phi}(T_t(x), h) | x \in cU, h \in \RR \}.
    \end{split}
\end{equation*}

Given $\mu \in \mathcal{M}(D')$ and $t \in cU$, we use this re-parametrization to define $\nu_t \in \mathcal{M}(D)$ as follows: for non-negative Borel functions $f$ on $\RR^{d+1}$,
\begin{equation}
\label{eq1026}
    \langle f, \nu_t \rangle : = \int f(\widetilde{\Phi}(T_t(x), h)) d\mu \circ \widetilde{\Phi}(x, h).
\end{equation}
Now we create a new measure $\nu \in \mathcal{M}(D)$ such that for non-negative Borel functions $f$ on $\RR^{d+1}$, 
\begin{equation*}
    \langle f, \nu \rangle : = \int_{\RR^{d-1}} \langle f, \nu_t\rangle \psi(t) dt,
\end{equation*}
where $\psi \in C_{0}^{\infty}(\RR^{d-1})$ with $\psi \geq 0$, $\text{spt} \psi \subset cU$, and $\psi(t) = 1$ for $t \in 2^{-1}cU$. $\nu$ is a weighted average of the measures $\nu_t$. Then
\begin{equation}
\label{eq1035}
    \widehat{\nu}(\xi) = \int_{\RR^{d-1}} \widehat{\nu_t}(\xi) \psi(t) dt.
\end{equation}

We have two lemmas about $\widehat{\nu}(\xi)$ when $\xi \| e_d$, where $e_d \in \RR^{d+1}$, $(e_d)_k = \delta_{d, k}$.

\begin{lemma}
\label{lem1042}
    If $|\widetilde{\eta}(t)| \geq 1$, $|\widehat{\nu}(k e_d)| \leq c_{\mu} \norm{\psi}_{L^1} k^{-\frac{s}{2}}.$ for $k>0$.
\end{lemma}

\begin{proof}
     For $k>0$, 
\begin{equation*}
\begin{aligned}
\widehat{\nu_t}(ke_d) &  = \int e^{-2 \pi i \widetilde{\Phi}(T_t(x), h) \cdot ke_d} d\mu \circ \widetilde{\Phi}(x, h) & \text{from (\ref{eq1026})}\\
&  = \int e^{-2 \pi i k Q_m(T_t(x)) } d\mu \circ \widetilde{\Phi}(x, h) & \text{by (\ref{eq991}), (\ref{eq998})}\\
 &  = e^{-2\pi i k \widetilde{\rho}(t)}\int e^{-2 \pi i k \widetilde{\Phi}(x, h) \cdot \widetilde{\eta}(t) } d\mu \circ \widetilde{\Phi}(x, h) & \text{by (\ref{eq1013})}\\
 &  = e^{-2\pi i k \widetilde{\rho}(t)} \int e^{-2 \pi i z \cdot k\widetilde{\eta}(t)} d\mu(z)\\
& = e^{-2\pi i k \widetilde{\rho}(t)} \widehat{\mu}(k\widetilde{\eta}(t)). 
\end{aligned}
\end{equation*} Then by (\ref{eq1035})
\begin{equation*}
    |\widehat{\nu}(ke_d)| = \left|\int_{\RR^{d-1}} e^{-2\pi i k \widetilde{\rho}(t)} \widehat{\mu}(k\widetilde{\eta}(t)) \psi(t)dt\right| \leq c_{\mu} \norm{\psi}_{L^1} k^{-\frac{s}{2}}.
\end{equation*}
\end{proof}

\begin{lemma}
\label{lem1065}
Assuming Lemma \ref{lem708}, there exist $c_0, k_0>0$ such that for $k \geq  k_0$, $|\widehat{\nu}(ke_d)| \geq c_0 k^{-\frac{d-1}{2}}$.
\end{lemma}

\begin{proof}   
By changing the order of integrations,
\begin{equation*}
\begin{aligned}
    \widehat{\nu}(ke_d) = & \int \int e^{-2 \pi i k Q_m(T_t(x)) } d\mu \circ \widetilde{\Phi}(x, h) & \text{from (\ref{eq1026}), (\ref{eq1035})} \\
    = & \int \conj{I(2\pi k; F_x,\psi)}d\mu \circ \widetilde{\Phi}(x,h),
\end{aligned}
\end{equation*}
where $F_x(t) =  Q_m(T_t(x))$, and $I$ is defined in (\ref{eq724}) of Theorem \ref{t61}. Suppose $F_x$ satisfies the conditions stated in Lemma \ref{lem708}. Let $c_{m}^{(d)} = i^{\frac{3(d-1)}{2}-m}$, and $H(x) = \det {\bf H} F_x(x)$, where the Hessian ${\bf H}$ is defined in (\ref{eq796}). We estimate the integrand $\conj{I(2\pi k; F_x,\psi)}$ using stationary phase (Theorem \ref{t61}) to obtain
\begin{equation*}
    I(2\pi k; F_x,\psi) =  c^{(d)}_m k^{-\frac{d-1}{2}}  |H(x)|^{-\frac{1}{2}}\psi(x) + E(x)k^{-\frac{d}{2}},
\end{equation*}
where $kh>\lambda_0$ for $\lambda_0 \in \RR$ sufficiently large and $E: cU \to \CC$ is a complex-valued function. In addition, $|E(x)| \leq E_0$ for $x \in cU$, where $E_0>0$ depends on $\lambda_0$, size of $cU$, $\sup\left\{\norm{F_x}_{C^{d+5}} | x \in cU\right\}$,  and $\norm{\psi}_{C^{d+2}}$. Therefore, we conjugate and integrate in $d\mu \circ \widetilde{\Phi}(x,h)$ to obtain
\begin{equation*}
\begin{aligned}
    & \int \conj{I(2\pi k; F_x,\psi)}d\mu \circ \widetilde{\Phi}(x,h) \\
    = & \conj{c^{(d)}_m} k^{-\frac{d-1}{2}}\int  |H(x)|^{-\frac{1}{2}} \psi(x) d\mu \circ \widetilde{\Phi}(x,h) + k^{-\frac{d}{2}}\int \conj{E(x)} d\mu \circ \widetilde{\Phi}(x,h),
\end{aligned}
\end{equation*} and
\begin{equation*}
\left| \widehat{\nu}(ke_d) - \conj{c^{(d)}_m} k^{-\frac{d-1}{2}}\int |H(x)|^{-\frac{1}{2}} \psi(x)  d\mu \circ \widetilde{\Phi}(x,h)  \right| \leq  E_0 k^{-\frac{d}{2}}\int d\mu \circ \widetilde{\Phi}(x,h).
\end{equation*} 
Suppose $\epsilon \leq  |H(x)| \leq \epsilon^{-1}$ for an $\epsilon > 0$. Since  $\psi(x) = 1$ for $x \in 2^{-1}cU$, $\mu$ is a probability measure supported on $D'$ ($D'$ is defined in (\ref{eq1004})), 
\begin{equation*}
    \int  |H(x)|^{-\frac{1}{2}} \psi(x) d\mu \circ \widetilde{\Phi}(x,h) \geq \epsilon^{-\frac{1}{2}},
\end{equation*}
and
\begin{equation*}
     \int  d\mu \circ \widetilde{\Phi}(x,h) = 1.
\end{equation*}
So $|\widehat{\nu}(ke_d)| \geq c_0 k^{-\frac{d-1}{2}}$ for $k \geq k_0$, where $c_0, k_0>0$ depend on $\epsilon, E_0$ and $d$.
\end{proof}

From these two lemmas, we have $s \leq d-1$. It remains to show that such $\{T_t\}$ exists in Section \ref{sec_ex}.

\section{Existence of the transformation $T_t$}
\label{sec_ex}

For $g: \RR^{n} \to \RR^{n}$, we denote ${\bf J} g$, the Jacobian of $g$, as the $n$ by $n$ matrix
\begin{equation}
\label{eq814}
    {\bf J} g := \begin{pmatrix}
        \frac{\partial g_1}{\partial y_1} & \cdots &\frac{\partial g_1}{\partial y_n} \\
        \vdots & \ddots & \vdots \\
        \frac{\partial g_n}{\partial y_1} & \cdots &\frac{\partial g_n}{\partial y_n} 
    \end{pmatrix}.
\end{equation}

For $G : \RR^n \times \RR^n \to \RR^n$  with inputs $x$ and $y$, we use
 ${\bf J}_x G(x, y)$ with the subscript $x$ to denote ${\bf J} G_y(x)$, where $G_y: \RR^n \to \RR^n$, $G_y(x) = G(x, y)$.

\begin{lemma}
\label{lem708}
\begin{enumerate}
    \item Let $U$ be a neighborhood of $0$. There exist $c> 0$ and a function $\eta: cU \to \RR^{d+1}$ with $|\eta(t)| \geq 1$. For each $t \in cU$, there exists a function $T_t: cU \to U$ satisfying the equation \begin{equation}
     Q_m(T_t(x))  = ( \phi(x), 1 ) \cdot \eta(t) \tag{\ref{eq6} revisited}.
\end{equation}
Here, $Q_m$ is defined in (\ref{eq764}) and $\phi$ is defined in (\ref{eq581}). Alternatively, the above equation can be written as \begin{equation}
     Q_m(T_t(x)) = \widetilde{\Phi}(x, h) \cdot \widetilde{\eta}(t) + \widetilde{\rho}(t). \tag{\ref{eq1013} revisited}
\end{equation} In the expression above, $\widetilde{\Phi}$ is defined in (\ref{eq998}), $\widetilde{\eta}: cU \to \RR^{d+1}$ with $|\widetilde{\eta}(t)| \geq 1$ and $\widetilde{\rho}: cU \to \RR$.
    \item \label{i717} Additionally, for each $x \in cU$, the function $F_x: cU \to \RR$ defined as $F_x(t) : = Q_m(T_t(x))$ satisfies the following three properties:
\begin{enumerate}
    \item $F_x(x)=0$.
    \item The gradient of $F_x$ with respect to $t$, $\nabla_t F_x(t)= 0$ when $t \in cU $ if and only if $t=x$.				
    \item \label{i721} ${\bf H} F_0(0) = {\bf H} Q_m(0)$, and there exists $ \epsilon > 0$, such that for $t \in cU$, $$|H(x)|=|\det {\bf H} F_x(t)| \in [\epsilon,  \epsilon^{-1}], $$
\end{enumerate}
where the Hessian ${\bf H}$ is defined in (\ref{eq796}).
\end{enumerate}
\end{lemma}

\begin{proof}
We first solve for $\eta$ and prove part \ref{i717}. Let $A_m : = {\bf H} Q_m(0)$, and $F: cU \times cU \to \RR$ be $$F(x, t) : = F_x(t) = ( \phi(x), 1 ) \cdot \eta(t).$$ Given $F(x, x)=0$ for $x \in cU$, $\nabla_t F(x, x)= 0$ is equivalent to $\nabla_x F(x, x)= 0$. Using the definition of $\phi$ in (\ref{eq581}), $\eta$ is a solution to the equation 
    \begin{equation*}
        \begin{pmatrix}
            \varphi(x)^T & Q_m(x) & 1 \\
            [{\bf J}\varphi(x)]^T & A_m x & 0
        \end{pmatrix} \eta(x) = 0,
    \end{equation*}
where the Jacobian ${\bf J}$ is defined in (\ref{eq814}). One solution is 
\begin{equation*}
    \eta(t) : = 
    \begin{pmatrix}
        -[{\bf J}\varphi(t)]^{-T} A_m t\\
        1 \\
        \varphi(t)^T[{\bf J}\varphi(t)]^{-T} A_m t -Q_m(t)
    \end{pmatrix},
\end{equation*}
where we use the notation $B^{-T} = (B^{-1})^{T}$ for an invertible matrix $B$. This choice of $\eta$ satisfies the requirement
$|\eta(t)| \geq 1$ and $F_x(x)=0$. We note that
\begin{equation}
\label{eq776}
    \begin{split}
    F(x, t) = & (\varphi(t)^T - \varphi(x)^T)[{\bf J}\varphi(t)]^{-T} A_m t + Q_m(x)-Q_m(t),\\
        \nabla_t F(x, t) = & w(x, t)+(\varphi(t)-\varphi(x))^T[{\bf J}\varphi(t)]^{-T}A_m,
    \end{split}
\end{equation}
where $w(x, t)\in \RR^{d-1}$, $(w(x, t))_j = (\varphi(t)-\varphi(x))^T \left[\frac{\partial}{\partial t_j} {\bf J}\varphi(t)\right]^{-T}A_m t$. Therefore,
$$\nabla_t F(x,x)=0, {\bf H}_t F(0, 0)=A_m.$$ 
The function $F_x$ satisfies the remaining properties in the list by the continuity of $F$ and the implicit function theorem if $c<1$ is chosen to be sufficiently small.

We can return to define the transformation $T_t$. First, we note that
\begin{equation*}
    \nabla_x F(x, t) = -t^T A_m [{\bf J} \varphi(t)]^{-1}[{\bf J} \varphi(x)]+x^T A_m,
\end{equation*}
and ${\bf H}_x F(0, 0) = A_m$. For $t\in cU$, let $$G(x, t) : =F(x+t, t).$$
We note that $G(0, t)=F(t,t)=0$, $\nabla_x G(0, t)=\nabla_x F(t, t) =0$, and ${\bf H}_x G(0, 0) = {\bf H}_x F(0, 0) = A_m$. Then, for each $t \in cU$, by Morse's lemma (Lemma \ref{lem32}), there exist $\widetilde{V}, \widetilde{W} \subset \RR^{d-1}$ that are neighborhoods of $0$, and a smooth $\tau: \widetilde{V} \times \widetilde{W} \to \RR^{d-1}$, such that $\tau(0, t)=0$, and for $x\in \widetilde{V}$, $$G(x, t)=Q_{m}(\tau(x, t)),$$ where $m$ is the number of positive eigenvalues of ${\bf H}_x G_t(0)= {\bf H}_x F(t, t)$. We note that $$G(x, 0)=F(x, 0)=Q_m(x)$$ by (\ref{eq776}), so $\tau(\cdot, 0)$ is the identity map on $\RR^{d-1}$. In addition, there exists a $c_{\tau}>0$, such that
$$|\tau(x, t)| = |\tau(x, t) - \tau(0, t)|\leq c_{\tau}|x|.$$

Then if $\{T_t\}$ exists and satisfies (\ref{eq6}),
$$Q_m(T_t(x)) = F(x, t)=G(x-t, t)=Q_m(\tau(x-t, t)),$$ and we can choose $T_t: cU \to U$ for $t \in cU$ as $$T_t(x) : = \tau(x-t, t),$$
which is well-defined if $c$ satisfies $2cU \subset \widetilde{V}$, $cU \subset \widetilde{W}$, and $2c<c_{\tau}^{-1}$.
    
Lastly, we write $\eta(t) = (\eta'(t), \widetilde{\rho}(t))$ for $\eta': cU \to \RR^{d}$ and $\widetilde{\rho}: cU \to \RR$. From the expression of the coordinate map $\widetilde{\Phi}$ in (\ref{eq998}), (\ref{eq1013}) holds if $\widetilde{\eta}: cU \to \RR^{d+1}$ is
    $\widetilde{\eta}(t) = (\eta'(t), 0)$.
\end{proof}

\section{A new question}
\label{sec_nq}
We propose a question related to the study of fractal sets.
\begin{conjecture}
     Let $A \subset \RR^{d}$ be a set with $\dim_F(A)=s \leq d$, not necessarily lying on a hypersurface, and 
     \begin{equation*}
         \begin{split}
             C_A := & \{h(x, 1) |x \in A, h \in \RR \} \subset \RR^{d+1} \\
             D_A := & S \times \RR = \{(x, h) |x \in A, h \in \RR \} \subset \RR^{d+1}
         \end{split}
     \end{equation*} be the cone and cylinder generated by $A$. Then $$\dim_F(C_A) = \dim_F(D_A) = \dim_F(A)=s.$$
\end{conjecture}
\begin{remark}
    By modifying the proof in section \ref{sec4}, we obtain $\dim_F(C_A), \dim_F(D_A) \geq s$. In addition, $s \leq \dim_H(A) \leq d$, so $s+1 \leq \dim_H(C_A), \dim_H(D_A) \leq d+1$.
\end{remark}

\section{Appendix}
\label{sec_app}

\subsection{Reduction to compactly supported measure}

We show a lemma that allows us to assume the measure $\mu \in \mathcal{M}(M)$ we study has smaller support in a smaller closed set on the manifold $M$.

\begin{lemma}\cite[Theorem 1]{eps}
\label{lem23}
     Suppose $\mu_0 \in \mathcal{M}(\RR^n)$, $\sup_{\xi \in \RR^{n}}|\xi|^{\alpha}|\widehat{\mu_0}(\xi)| < \infty$ for $\alpha >0$. Let $f \in \mathcal{S}(\RR^{n})$ with $f \geq 0$, and $\mu \in \mathcal{M}(\RR^n)$ such that $d\mu = fd\mu_0$. Then  $|\widehat{\mu}(\xi)| \leq c_{\mu} |\xi|^{-\alpha}$ for a $c_{\mu}>0$.
\end{lemma}

\begin{proof}
    Note that
    \begin{equation*}
    \begin{split}
        |\widehat{\mu}(\xi)| & = |\widehat{\mu_0} * \widehat{f}(\xi)| \\
        & = \left|\int \widehat{\mu_0}(\xi - \eta) \widehat{f}(\eta) d\eta  \right| \\
        & \leq \left|\int_{|\eta| \leq \frac{|\xi|}{2}} \widehat{\mu_0}(\xi - \eta) \widehat{f}(\eta) d\eta  \right| + \left|\int_{|\eta| \geq \frac{|\xi|}{2}} \widehat{\mu_0}(\xi - \eta) \widehat{f}(\eta) d\eta  \right|.
    \end{split}
    \end{equation*}

    For the first integral, since $|\eta| \leq \frac{|\xi|}{2}$, $|\xi-\eta| \geq \frac{|\xi|}{2}$, and $\widehat{f}$ is integrable, the integral is bounded above by a constant multiple of $|\xi|^{-\alpha}$. For the second integral, we apply the bounds $\norm{\widehat{\mu_0}}_{L^{\infty}}=1$ and $\int_{|\eta| \geq \frac{|\xi|}{2}} |\widehat{f}(\eta)| d\eta \leq c_m |\xi|^{-m}$ for $m \in \NN$ with a $c_m > 0$ since $\widehat{f} \in \mathcal{S}(\RR^n)$. If we choose $m > \alpha$, the sum of two bounds is bounded by a constant multiple of $|\xi|^{-\alpha}$ for large $|\xi|$.
\end{proof}

\subsection{Gaussian curvature and Morse's lemma}
\label{sec62}
In differential geometry, a way to describe $M$ is through notions of curvatures, which is defined through the eigenvalues of the \textit{Weingarten map} \cite{edg}. 
Let $T_pM$ be the tangent space of $M$ at $p \in M$. The \textit{Weingarten map} $L_p: T_p M \to T_p M$ at $p \in M$ is the linear map $L_p(v) = -D_v N = -\frac{d}{dt} (N \circ \gamma) (0)$, where $\gamma: I \to M$ is a curve with $\gamma(0)=p$, $\gamma'(0)=v$. The \textit{principal curvatures} of $M$ at $p$ are the eigenvalues of the map $L_p$, and the \textit{Gaussian curvature} of $M$ at $p$ is the product of the eigenvalues, which equals the determinant of $L_p$. We note that Gaussian and principal curvatures at $p \in M$ are independent of the parametrization of $M$ and the choice of a basis for $T_p M$. 

The following lemma provides a way to compute the Gaussian curvature and principal curvatures at points $p \in M$ when $M$ is the graph of a function $g: \RR^{d-1} \to \RR$ with $g(0)=0$, $\nabla g(0)=0$.
\begin{lemma}[{\cite[Section 8.3]{Stein1993}}]
\label{lem63}
Let $M$ be an open subset of a smooth $d-1$ dimensional submanifold of $\RR^d$. For $p \in M$, by rotation and translation, $p$ is moved to the origin, and the tangent plane of $M$ at $p$ becomes the hyperplane $x_d = 0$. Near the origin, the surface $M$ can be given as a graph of $$x_d = g(x'),$$ where $x' \in \RR^{d-1}$, with $g \in C^{\infty}(V)$ for open $V\subset \RR^{d-1}$ that contains $0$, and $\phi(0) = 0$, $\nabla \phi(0) = 0$. The \textit{principal curvatures} of the point $p \in M$ are the eigenvalues of the matrix 
${\bf H}g(0).$
The \textit{Gaussian curvature} is the product of the eigenvalues, which equals the determinant of the matrix above.
\end{lemma}

Morse's lemma is a tool to study hypersurfaces with non-vanishing Gaussian curvature. The version stated below generalizes the one shown in \cite[Section 8.2]{Stein1993}.

\begin{lemma}
\label{lem32}
Suppose $f \in C^{\infty}(\RR^d \times \RR^{d})$, and for each $t \in \RR^{d}$, $f_t \in C^{\infty}(\RR^d)$ given by $f_t(x) = f(x, t)$ has a non-degenerate critical point at $x=0$, which means $\nabla_x f(0, t) = 0$ and ${\bf H}_xf(0, t)$ is invertible. We further assume that $f(0, t)=0$. Then there exist neighbourhoods $V, W$ of $0$ and a smooth $\tau: V \times W \to \RR^{d}$ such that $\tau(0, t)=0$, $\det {\bf J}_x \tau(x, t) \neq 0$, and
\begin{equation}
\label{eq803}
    f(x, t) = Q_m(\tau(x, t)),
\end{equation}
where $Q_m$ is defined in (\ref{eq764}) with $0 \leq m \leq d$.
\end{lemma}

\begin{remark}
    The number of positive eigenvalues of the matrix ${\bf H}_xf(0, 0)$ is $m$. By differentiating (\ref{eq803}) twice,
\begin{equation}
    \label{eq828}
     [{\bf J}_x\tau(0, t)]^T{\bf H}Q_m(0)  {\bf J}_x\tau(0, t) = {\bf H}_x f(0, t).
\end{equation}
\end{remark}

\begin{proof}
We claim that the function $\tau$ can be expressed as
\begin{equation}
\label{mor869}
    \tau = L \circ \tau_d \circ \cdots \circ \tau_{1},
\end{equation}
where each $\tau_r$ is a change of variables in the first $d$ coordinates, and $L$ is a permutation of the same coordinates. $\tau_r$ is constructed inductively as follows: suppose at step $r$, we have
$$\widetilde{g_r}(x^{(r)}, t) := f(\tau_{1}^{-1}\circ \cdots \circ \tau_{r-1}^{-1}(x^{(r)}, t)) = \pm [x^{(r)}_1]^2 \pm \cdots \pm [x^{(r)}_{r-1}]^2+ \sum_{j, k \geq r}^{d} x^{(r)}_j x^{(r)}_k \widetilde{f^{(r)}_{j, k}}(x^{(r)}, t),$$
where $$\widetilde{f^{(r)}_{j, k}}(x^{(r)}, t)= \int_{0}^{1}(1-s) \frac{\partial^2 \widetilde{g_r}}{\partial x^{(r)}_j\partial x^{(r)}_k}(sx^{(r)}, t)ds, \widetilde{f^{(r)}_{j, k}}(0, t) = \frac{1}{2} \frac{\partial^2  \widetilde{g_r}}{\partial x^{(r)}_j \partial x^{(r)}_k}(0, t).$$ Then there exists an orthonormal matrix $O_r$, which is a linear change in the variables $x^{(r)}_r, \cdots, x^{(r)}_d$, such that 
$$g_r(y^{(r)},t) :=  f(\tau_{1}^{-1}\circ \cdots \circ \tau_{r-1}^{-1}(O_r y^{(r)}, t)) = \pm [y^{(r)}_1]^2 \pm \cdots \pm [y^{(r)}_{r-1}]^2+ \sum_{j, k \geq r}^{d} y^{(r)}_j y^{(r)}_k f^{(r)}_{j, k}(y^{(r)}, t),$$
where $$f^{(r)}_{j, k}(y^{(r)}, t)= \int_{0}^{1}(1-s) \frac{\partial^2 g_r}{\partial y^{(r)}_j\partial y^{(r)}_k}(sy^{(r)}, t)ds, f^{(r)}_{j, k}(0, t) = \frac{1}{2} \frac{\partial^2  g_r}{\partial y^{(r)}_j \partial y^{(r)}_k}(0, t)$$ and the additional condition that $f^{(r)}_{r, r}(0, t)  \neq 0$. Then, for each $t$, we can perform a change of variables from $y^{(r)}$ to $z^{(r)}$ such that $z^{(r)}_j = y^{(r)}_j$ for $j \neq r$, and
\begin{equation}
\label{mor875}
    z^{(r)}_r = [\pm f^{(r)}_{r,r}(y^{(r)}, t)]^{\frac{1}{2}}\left[y^{(r)}_r + \sum_{j > r} \frac{y^{(r)}_j
f^{(r)}_{j, r}(y^{(r)}, t)}{\pm f^{(r)}_{r, r}(y^{(r)}, t) } \right],
\end{equation}
where $\pm$ is the sign of $f^{(r)}_{r, r}(0, t)$. This change of variables can be expressed as $z^{(r)}=\sigma_{r}(y^{(r)}, t)$ for $\sigma_r: V_r \times W_r \to \RR^{d}$, where $V_r, W_r$ are neighborhoods of $0$, $ f^{(r)}_{r,r}$ does not change sign on $V_r \times W_r$, and $$\det {\bf J}_{y^{(r)}} \sigma_{r}(y^{(r)}, t) \neq 0.$$ Then we let 
\begin{equation}
    \label{mor884}
    \tau_r(x^{(r)}, t)= (\sigma_{r}(O_r^{-1} x^{(r)}, t), t),
\end{equation}
and proceed to the next step by noting $\tau_r^{-1}$ exists in a neighborhood of $0$ and taking $x^{(r+1)}=z^{(r)}$.

From the construction above, there exist $V, W$  neighbourhoods of $0$, such that for $(x, t) = (x^{(1)}, t) \in V \times W$, for all $r$ from $1$ to $d$,  $f^{(r)}_{r,r}(y^{(r)}, t)$ does not change sign, and $$\det {\bf J}_{y^{(r)}} \sigma_{r}(y^{(r)}, t) \neq 0.$$ So $\tau$ is defined on $V \times W$. 
\end{proof}

Let $\tau_t(x) = \tau(x, t)$. For $k \geq 0$, $t \in W$, it is possible to bound $\norm{{\bf J}\tau_t}_{C^{k}(V)}$ and $\inf\{|\det {\bf J  }\tau_t(x)| x \in V \}$ using $\norm{f_t}_{C^{k+2}}$, $|\det {\bf J}\tau_t(0)|$, and the size of $V$ using (\ref{mor869}), (\ref{mor875}), and (\ref{mor884}).

\subsection{Oscillatory Integrals}

We refer readers to the complete proof of oscillatory integrals results in \cite{Stein1993}. We outline the key steps and bounds on the error terms in this section.

For two functions $f, g: D \to \RR_{\geq 0}$ with a domain $D$, we write $f \lesssim g$ to denote that there exists a $c>0$, such that for $x\in D$, $f(x) \leq c g(x)$.

\begin{lemma}\cite[Chapter 8, Proposition 6]{Stein1993} Let $Q_m(y)$ be defined in (\ref{eq764}).
    \begin{enumerate}[a.]
    \item  If $\eta \in C^{\infty}_{0}(\RR^{n})$,
        \begin{equation}
        \label{eq32}
            \left|\int_{\RR^{n}} e^{i \lambda Q_m(y)} y^{l} \eta(y) dy \right| \leq A_l \lambda^{-\frac{n+|l|}{2}},
        \end{equation}
        where $A_l>0$, $l \in \ZZ^{n}$, $l_j \geq 0$, and $|l| = \sum_{j=1}^{n} l_j$.
    \item  If $g \in \mathcal{S}(\RR^{n})$ and there exists $\delta > 0$, such that $g(y)=0$ for $y \in B(0, \delta)$, then for $N \in \NN$, there exists $B_N > 0$ such that
        \begin{equation}
        \label{eq836}
            \left|\int_{\RR^{n}} e^{i \lambda Q_m(y)} g(y) dy \right| \leq B_N \lambda^{-N}.
        \end{equation}
    \end{enumerate}
\end{lemma}

\begin{proof}
\begin{enumerate}[a.]
    \item Consider the cones 
    $$\Gamma_k = \left\{y \in \RR^{n}| |y_k|^2 \geq \frac{|y|^2}{2n}  \right\},$$
    and  $$\Gamma_k^{0} = \left\{y \in \RR^{n}| |y_k|^2 \geq \frac{|y|^2}{n}  \right\}.$$
    Since $\cup_{k=1}^{n} \Gamma_k^{0} = \RR^{n}$, there are functions $\{\Omega_k\}_{1 \leq k \leq n}$ such that each $\Omega_k$ is homogeneous of degree $0$, smooth away from the origin, $0 \leq \Omega_k \leq 1$ with $$\sum_{k=1}^{n} \Omega_k(x) = 1$$ for $x\neq 0$, and each $\Omega_k$ is supported in $\Gamma_k$. Then
    \begin{equation*}
        \int_{\RR^{n}} e^{i \lambda Q_m(y)} y^{l} \eta(y) dy = \sum_{k=1}^{n} \int_{\Gamma_{k}} e^{i \lambda Q_m(y)} y^{l} \eta(y) \Omega_k(y) dy.
    \end{equation*}
    
     In the cone $\Gamma_k$, we will show that there exists $A_{l, k}>0$ such that 
     \begin{equation}
     \label{eq797}
       \left| \int_{\Gamma_{k}} e^{i \lambda Q_m(y)} y^{l} \eta(y) \Omega_k(y) dy\right| \leq A_{l, k} \lambda^{-\frac{n+|l|}{2}}.  
     \end{equation}
     Then summing over all $k$ yields (\ref{eq32}). Let $\alpha \in C^{\infty}(\RR^n)$, such that $\alpha(y) = 1$ for $|y| \leq 1$, and $\alpha(y)=0$ for $|y| \geq 2$. Then for $\epsilon>0$,
        \begin{equation*}
            \begin{split}
            & \int_{\Gamma_{k}} e^{i \lambda Q_m(y)} y^{l} \eta(y) \Omega_k(y) dy  \\
            = & \int_{\Gamma_{k}} e^{i \lambda Q_m(y)} y^{l} \eta(y) \Omega_k(y) \alpha(\epsilon^{-1} y) dy + \int_{\Gamma_{k}} e^{i \lambda Q_m(y)} y^{l} \eta(y) \Omega_k(y) [1-\alpha(\epsilon^{-1} y)] dy.
            \end{split}    
        \end{equation*}
        For the first integral,
        \begin{equation*}
            \left|\int_{\Gamma_{k}} e^{i \lambda Q_m(y)} y^{l} \eta(y) \Omega_k(y) \alpha(\epsilon^{-1} y) dy  \right| \lesssim \norm{\eta}_{L^{\infty}} \epsilon^{|l|+1}.
        \end{equation*}
        Let $N \in \NN$. Using integration by parts $N$ times, the second integral can be written as
        \begin{equation}
        \label{eq35}
             \int_{\Gamma_k} e^{i \lambda Q_m(y)} (^tD_k)^N\{ y \eta(y)  \Omega_k(y)  [1-\alpha(\epsilon^{-1} y)] \}dy
        \end{equation}
        with $^tD_k g = s_k(2 i\lambda)^{-1} \frac{\partial}{\partial y_k}\left(\frac{g}{y_k}\right)$ for a differentiable function $g$, and $s_k=-1$ if $k \leq m$, $s_k=1$ if $k \geq m+1$. Note that
        \begin{equation}
        \label{eq36}
            (^tD_k)^N g = \lambda^{-N}\sum_{r=0}^{N}a^{(m, k)}_{N, r} y_k^{r-2N}\frac{\partial^r g}{\partial y_k^r}
        \end{equation} for $a^{(m, k)}_{N, r} \in \CC$. When we apply (\ref{eq36}) and the product rule of the derivative to expand (\ref{eq35}), we obtain a summation of terms where a term, ignoring the constant, is
        \begin{equation*}
                \lambda^{-N}\int_{\Gamma_k \cap B(0, \epsilon)^{c}} e^{i \lambda Q_m(y)} y_k^{r-2N} \left[\frac{\partial^{r_1}}{\partial y_k^{r_1}} y^{l} \right]  \left[\frac{\partial^{r_2}}{\partial y_k^{r_2}}\eta(y)\right]\left[\frac{\partial^{r_3}}{\partial y_k^{r_3}}\Omega_k(y)\right]\left[\frac{\partial^{r_4}}{\partial y_k^{r_4}} [1-\alpha(\epsilon^{-1} y)] \right]dy
        \end{equation*}
        for $r_1, r_2, r_3, r_4 \geq 0$, $r_1+r_2+r_3+r_4 = r \leq N$. We note that $\frac{\partial^{r_3} \Omega_k}{\partial y_k^{r_3}}$ is a homogeneous function of degree $-r_3$. When $|l|-N < -n$, the term above is bounded by a constant multiple of 
        \begin{equation*}
        \begin{split}
            & \lambda^{-N} \epsilon^{-r_4}\int_{\Gamma_k \cap B(0, \epsilon)^{c}} |y|^{r-2N+|l|-r_1-r_3} \norm{\frac{\partial^{r_2}\eta}{\partial y_k^{r_2}}}_{L^{\infty}} \norm{\frac{\partial^{r_4} \alpha}{\partial y_k^{r_4}}}_{L^{\infty}} dy \\
            \lesssim &   \lambda^{-N} \epsilon^{-r_4} \epsilon^{r-2N+|l|-r_1-r_3+n} \norm{\eta}_{C^{r_2}}\norm{\alpha}_{C^{r_4}}\\
            \lesssim & \lambda^{-N} \epsilon^{|l|-2N+r-r_1-r_3-r_4+n} \norm{\eta}_{C^{r_2}}\norm{\alpha}_{C^{r_4}}.
        \end{split}
        \end{equation*}
        Then we obtain (\ref{eq797}) by setting $\epsilon = \lambda^{-\frac{1}{2}}$. $A_k$ depends on $\norm{\eta}_{C^{N}}$.
    \item The proof for b) is similar to a). We use the same cone $\Gamma_k$ and functions $\Omega_k$, $\alpha$. It suffices to show
    \begin{equation}
    \label{eq896}
        \left|\int_{\Gamma_k} e^{i \lambda Q_m(y)} g(y) \Omega_k(y) dy  \right| \leq B_{N, k} \lambda^{-N}\
    \end{equation}
    for a $B_{N, k}>0$. If $2\epsilon < \delta$, 
     \begin{equation*}
             \int_{\Gamma_{k}} e^{i \lambda Q_m(y)} g(y) \Omega_k(y) dy 
            =  \int_{\Gamma_k \cap B(0, \epsilon)^{c}} e^{i \lambda Q_m(y)} g(y) \Omega_k(y) [1-\alpha(\epsilon^{-1} y)] dy.   
        \end{equation*}
    Then we apply integration by parts $N$ times to obtain
    \begin{equation}
    \label{eq905}
          \int_{\Gamma_k \cap B(0, \epsilon)^{c}} e^{i \lambda Q_m(y)} (^tD) \left\{g(y) \Omega_k(y) [1-\alpha(\epsilon^{-1} y)] \right\}dy.
    \end{equation}
    After we apply (\ref{eq36}) and the product rule of the derivative to expand (\ref{eq905}), we obtain a summation of terms where a term, ignoring the constant, is
        \begin{equation*}
                \lambda^{-N}\int_{\Gamma_k} e^{i \lambda Q_m(y)} y_k^{r-2N}   \left[\frac{\partial^{r_1}}{\partial y_k^{r_1}}g(y)\right]\left[\frac{\partial^{r_2}}{\partial y_k^{r_2}}\Omega_k(y)\right]\left[\frac{\partial^{r_3}}{\partial y_k^{r_3}} [1-\alpha(\epsilon^{-1} y)] \right]dy
        \end{equation*}
        for $r_1, r_2, r_3 \geq 0$, $r_1+r_2+r_3 = r \leq N$. We note that $\frac{\partial^{r_2} \Omega_k}{\partial y_k^{r_2}}$ is a homogeneous function of degree $-r_2$. When $-N < -n$, the term above is bounded by a constant multiple of 
        \begin{equation*}
        \begin{split}
            & \lambda^{-N} \epsilon^{-r_3}\int_{\Gamma_k \cap B(0, \epsilon)^{c}} |y|^{r-2N-r_2} \norm{\frac{\partial^{r_1}g}{\partial y_k^{r_1}}}_{L^{\infty}} \norm{\frac{\partial^{r_3} \alpha}{\partial y_k^{r_3}}}_{L^{\infty}} dy \\
            \lesssim &   \lambda^{-N} \epsilon^{-r_3} \epsilon^{r-2N-r_2+n} \norm{g}_{C^{r_1}}\norm{\alpha}_{C^{r_3}}\\
            \lesssim & \lambda^{-N} \epsilon^{r-2N-r_2-r_3+n} \norm{g}_{C^{r_1}}\norm{\alpha}_{C^{r_3}}.
        \end{split}
        \end{equation*}
        Then we obtain (\ref{eq896}) by setting $\epsilon = \frac{\delta}{3}$. 
    \end{enumerate}
\end{proof}

For $f \in C^{\infty}(\RR^{n})$, $U$ open in $\RR^{n}$, and $k\geq 0$, we denote $\norm{f}_{C^k(U)}$, or $\norm{f}_{C^k}$ if the implication of the open set $U$ is clear, as the quantity
$$\sum_{|\beta| \leq k} \norm{\frac{\partial^{|\beta|}}{\partial y^{\beta}} f}_{L^{\infty}(U)}.$$

\begin{theorem}\cite[Chapter 8, Proposition 6]{Stein1993}
\label{t61}
    \begin{enumerate}[a.]
        \item Let $Q_m(y)$ be defined as in (\ref{eq764}). Let
        $$I_m(\lambda; \psi) : = \int_{\RR^n} e^{i \lambda Q_m(y)} \psi(y) dy,$$ where $\psi$ is supported in a small neighborhood of $0$, and $\lambda_0 > 1$. For $\lambda \geq \lambda_0$,
        \begin{equation}
        \label{eq28}
            \left|I_m(\lambda; \psi) - (-1)^{\frac{n-m}{2}}(\pi i)^{\frac{n}{2}} \psi(0) \lambda^{-\frac{n}{2}}\right| \leq D\lambda^{-\frac{n+1}{2}},
        \end{equation}
        where $D$ depends on $\lambda_0$, the size of $\text{spt } \psi$, and $\norm{\psi}_{C^{n+3}}$.
        
        \item Let $\phi, \psi \in C^{\infty}(\RR^{n})$. Suppose $\phi$ only has one critical point $z_0$ in $\text{spt } \psi$, which is non-degenerate, and $\phi(z)=0$. Suppose $\text{spt } \psi \subset V$, where $V$ is a neighborhood of $z_0$ obtained from applying Morse's lemma (Lemma \ref{lem32}) to $\phi$. Let
        \begin{equation}
        \label{eq724}
            I(\lambda; \phi, \psi) : = \int_{\RR^n} e^{i \lambda \phi(z)} \psi(z) dz,
        \end{equation}
        $\lambda_0 > 1$, and $c=\left|\det {\bf H} \phi(z_0)\right|$, where the Hessian ${\bf H}$ is defined in (\ref{eq796}). Then for $\lambda \geq \lambda_0$, there exist $0\leq m \leq n$ and $D=D(\lambda_0, \phi, \psi)>0$, such that
        \begin{equation}
        \label{eq858}
            \left|I(\lambda; \phi, \psi) - (-1)^{\frac{n-m}{2}}(2\pi i)^{\frac{n}{2}} c^{-\frac{1}{2}} \psi(z_0) \lambda^{-\frac{n}{2}}\right| \leq D(\lambda_0, \phi, \psi)|\lambda|^{-\frac{n+1}{2}}.
        \end{equation}
        The error $D(\lambda_0, \phi, \psi)$ depends on $\lambda_0$, $c$, the size of $\text{spt } \psi$, and $\norm{\phi}_{C^{n+6}}$, $\norm{\psi}_{C^{n+3}}$.
    \end{enumerate}
\end{theorem}

\begin{proof}

    \begin{enumerate}[a.]
        \item Step 1: We have
        \begin{equation*}
            \int_{-\infty}^{\infty} e^{i \lambda x^2} e^{-x^2} dx = (1-i\lambda)^{-\frac{1}{2}} \int_{-\infty}^{\infty} e^{-x^2} dx,
        \end{equation*}
        and $\int_{-\infty}^{\infty} e^{-x^2} dx = \sqrt{\pi}$. We can fix the principal branch of $z^{-\frac{1}{2}}$ in the complex plane slit along the negative real-axis. Therefore,
        \begin{equation}
        \label{eq895}
        \begin{split}
            \int_{\RR^{n}} e^{i \lambda Q_m(y)} e^{-|y|^2} dy 
            = & \left(\prod_{j=1}^{m}\int_{\RR} e^{i \lambda y_j^2} e^{-y_j^2} dy_j\right)\left(\prod_{j=m+1}^{n}\int_{\RR} e^{-i \lambda y_j^2} e^{-y_j^2} dy_j\right) \\
            = & \pi^{\frac{n}{2}}(1-i\lambda)^{-\frac{m}{2}}(1+i\lambda)^{-\frac{n-m}{2}} \\
            = & \pi^{\frac{n}{2}} \lambda^{-\frac{n}{2}}(\lambda^{-1}-i)^{-\frac{m}{2}}(\lambda^{-1}+i)^{-\frac{n-m}{2}}.
        \end{split}
        \end{equation}
        
        We write the power series expansion of $f_m(w) = (w-i)^{-\frac{m}{2}}(w+i)^{-\frac{n-m}{2}}$ at $0$ as
        $\sum_{k=0}^{\infty} a_k w^k,$
        where $a_0 = (-1)^{\frac{n-m}{2}} i^{\frac{n}{2}}$. Let $\gamma$ be a line segment from $0$ to $\lambda^{-1}$, and $b_m =  \sup_{|w| = \lambda_0^{-1}} |f_m'(w)|$. Then for $\lambda \geq \lambda_0 > 1$, the error of approximation by the constant term is bounded by
        \begin{equation}
        \label{eq800}
        \begin{split}
            |f_m(\lambda^{-1}) -  a_0| \leq & \left|\int_{\gamma} f_m'(w) dw \right|\\
            \leq & |\gamma| \sup_{|w| = \lambda^{-1}} |f_m'(w)| \\
            \leq & b_m \lambda^{-1} 
        \end{split} 
        \end{equation}
        using the Integral form of the Taylor remainder and the maximum modulus principle.  Putting everything together, for $\lambda \geq \lambda_0$,
        \begin{equation}
        \label{eqn_32}
            \begin{aligned}
                 & \left|\int_{\RR^{n}} e^{i \lambda Q_m(y)} e^{-|y|^2} dy - a_0 \pi^{\frac{n}{2}} \lambda^{-\frac{n}{2}} \right| \\
             = & \left| \pi^{\frac{n}{2}} \lambda^{-\frac{n}{2}}(\lambda^{-1}-i)^{-\frac{m}{2}}(\lambda^{-1}+i)^{-\frac{n-m}{2}}  -  a_0 \pi^{\frac{n}{2}} \lambda^{-\frac{n}{2}}\right| & \text{ by (\ref{eq895})}\\
             = & \pi^{\frac{n}{2}} \lambda^{-\frac{n}{2}} \left|  (\lambda^{-1}-i)^{-\frac{m}{2}}(\lambda^{-1}+i)^{-\frac{n-m}{2}}  -  a_0 \right| \\
            \leq & b_m \pi^{\frac{n}{2}} \lambda^{-\frac{n+2}{2}}  & \text{ by (\ref{eq800})}.
            \end{aligned}
        \end{equation}

    Step 2: To obtain (\ref{eq28}), we write $e^{|y|^2}\psi(y) = \psi(0) + \sum_{y=1}^{n} y_j R_j(y)$, where $R_j \in C_0^{\infty}(\RR^n)$. Let $\bar{\psi} \in C^{\infty}_0(\RR^n)$, with $\bar{\psi}(y) = 1$ on the support of $\psi$. To apply results from the previous steps, we write
    \begin{equation*}
        \begin{split}
            & \int_{\RR^{n}} e^{i \lambda Q_m(y)} \psi(y) dy \\ = & \int e^{i \lambda Q_m(y)} e^{-|y|^2} [e^{|y|^2} \psi(y)] \bar{\psi}(y) dy \\
            = & \int_{\RR^{n}} e^{i \lambda Q_m(y)} e^{-|y|^2} \left[\psi(0)+\sum_{j=1}^{n} y_j R_j(y)\right] \bar{\psi}(y) dy \\
            = & \int e^{i \lambda Q_m(y)} e^{-|y|^2} \psi(0) \bar{\psi}(y) dy + \sum_{j=1}^{n}\int e^{i \lambda Q_m(y)} y_j e^{-|y|^2} R_j(y) \bar{\psi}(y) dy \\
            = & \psi(0) \int e^{i \lambda Q_m(y)} e^{-|y|^2}  dy+ \psi(0) \int e^{i \lambda Q_m(y)} e^{-|y|^2}  [1- \bar{\psi}(y) ]dy + \sum_{j=1}^{n} \int e^{i \lambda Q_m(y)} y_j e^{-|y|^2} R_j(y) \bar{\psi}(y) dy.
        \end{split}
    \end{equation*}
    
    Then we obtain (\ref{eq28}) by applying (\ref{eqn_32}) to the first integral, (\ref{eq836}) to the second integral, and (\ref{eq32}) to each term in the summation. We note that the error $D$ depends on $\lambda_0$, the size of the support of $\psi$, and $\norm{\psi}_{C^{n+3}}$.
    \item 
    By Morse's lemma (Lemma \ref{lem32}),
    there exists a diffeomorphism $\tau: V \to U$, where $V$ is a neighborhood of $z_0$ in the $z$-space, and $U$ is a neighborhood of $0$ in the $y$-space, such that $\tau(z_0)=0$ and 
    \begin{equation*}
        \phi(z) = Q_m(\tau(z)).
    \end{equation*} 
    By a change of variables $z=\tau^{-1}(y)$,
    $$ \int_{\RR^n} e^{i \lambda \phi(z)} \psi(z) dz =  \int e^{i \lambda Q_m(y)}\psi(\tau^{-1}(y)) |\det {\bf J  }\tau^{-1}(y)|dy=I_m(\lambda; (\psi \circ \tau^{-1})\cdot |\det {\bf J  }\tau^{-1}|),$$
    where ${\bf J}\tau$ is the Jacobian of $\tau$ (defined in (\ref{eq814})).
    We can apply part a) to the integral above. $D$ depends on $\lambda_0$, the size of support of $\psi \circ \tau^{-1}$, and $C^{n+3}$ norms of $(\psi \circ \tau^{-1})\cdot |\det {\bf J  }\tau^{-1}|$. We note that the $L^{\infty}$ norm of the $k^{\text{th}}$ partial derivative ($0 \leq k \leq n+3$) of $(\psi \circ \tau^{-1})\cdot |\det {\bf J  }\tau^{-1}|$ can be bounded by $C^{n+3}$ norms of $\psi$, $|\det {\bf J  }\tau|$, and $\inf\{|\det {\bf J  }\tau(z)| z \in \text{spt} \psi  \}$. We note that $|\det {\bf J  }\tau(0)|^2=\frac{c}{2^n}$ by (\ref{eq828}), and we can bound the error in terms of $\phi$ with the fact that $\norm{\det {\bf J  }\tau}_{C^{n+3}(V)}$  and $\inf\{|\det {\bf J  }\tau(z)| z \in \text{spt} \psi  \}$ depend on $\norm{\phi}_{C^{n+5}(V)}$, $c$, and the size of $\text{spt } \psi$ from the discussion of Morse's lemma in section \ref{sec62}.
    \end{enumerate}
\end{proof}

\bibliographystyle{plain} 
\bibliography{refs}

\Addresses

\end{document}